\documentclass[a4paper,11pt,reqno]{amsart}
\usepackage{amssymb}
\usepackage{amsmath}
\usepackage{amscd}
\usepackage{stmaryrd}
\usepackage{amsbsy}
\usepackage{multirow}
\usepackage{tikz}
\usepackage{tikz-cd}
\usepackage{enumitem}
\usepackage{adjustbox}

\usepackage{color}

\def\a{\alpha}
\def\b{\beta}
\def\c{\gamma}
\def\f{\varphi}

\def\lg{\langle}
\def\rg{\rangle}
\def\ra{\rightarrow}
\def\lra{\longrightarrow}
\def\.{\cdot}

\def\nb{\nabla}
\def\l{\lambda}
\def\tl{\tilde}
\def\beq{\begin{equation}}
\def\eeq{\end{equation}}
\def\bi{\begin{enumerate}}
\def\ei{\end{enumerate}}
\def\bea{\begin{eqnarray*}}
\def\eea{\end{eqnarray*}}
\def\ba{\begin{array}}
\def\ea{\end{array}}
\def\bpm{\begin{pmatrix}}
\def\epm{\end{pmatrix}}
\def\A{\mathcal{A}}

\def\L{\Lambda}
\def\k{\kappa}

\def\r{\end{proof}}

\def\hyp{{H}}
\def\ell{{E}}
\def\para{{P}}
\def\Hyp{\mathcal{H}}
\def\Ell{\mathcal{E}}
\def\Para{\mathcal{P}}
\def\Aut{\mathrm{Aut}}
\def\pj{\mathbf{p}}

\def\ot{\otimes}

\def\ci{\mathcal{C}^\infty}

\def\m{\mathfrak{m}}

\def\dg{\dot\gamma}

\def\m{\mathfrak{m}}

\def\M{\mathcal{M}}

\def \R{\mathbb{R}}

\def \N{\mathbb{N}}
\def \Z{\mathbb{Z}}

\def \T{\mathrm{T}}
\def \H{\mathbb{H}}

\def\9{[\! (}
\def\0{)\! ]}
\def\[{\pmb{[}}
\def\]{\pmb{]}}

\def\lra{\longrightarrow}
\def\id{\mathrm{id}}
\def\be{\begin{equation}}

\def\ee{\end{equation}}

\def\LL{{\mathcal L\,}}

\def\tr{\mathrm{tr}}

\def\GL{\mathrm{GL}}
\def\PGL{\mathrm{PGL}}
\def\PSL{\mathrm{PSL}}
\def\SL{\mathrm{SL}}
\def\rp{\mathbb{R}\mathbb{P}}

\def\SO{\mathrm{SO}}

\def\Sym{\mathrm{Sym}}
\def\Scal{\mathrm{Scal}}

\def\tp{\widetilde{\mathbb{R}\mathbb{P}^1}}
\def\tsl{\widetilde{\SL}(2,\mathbb{R})}
\def\sm{\smallsetminus}

\def\Hss{\mathrm{Hess}}

\newtheorem{ede}{Definition}[section]

\newtheorem{cor}[ede]{Corollary}%[chapter]
\newtheorem{prop}[ede]{Proposition}%[chapter]
\newtheorem{lem}[ede]{Lemma}%[chapter]
\newtheorem{thrm}[ede]{Theorem}%[chapter]

\newtheorem{ath}[ede]{Theorem}

\newtheorem{elem}[ede]{Lemma}

\theoremstyle{definition}
\newtheorem{remark}[ede]{Remark}
\def\obs{\begin{definition}}
\def\eobs{\end{definition}}
 
\title{Projective structures on curves and conformal geometry}
\author{Florin Belgun, Andrei Moroianu}

\address{Florin Belgun, Institute of Mathematics ``Simion Stoilow'' of the Romanian Academy, 21 Calea Grivitei, 010702 Bucharest, Romania}
 \email{fbelgun@imar.ro}
 
\address{Andrei Moroianu, Université Paris-Saclay, CNRS,  Laboratoire de mathématiques d'Orsay, 91405 Orsay, France, and Institute of Mathematics ``Simion Stoilow'' of the Romanian Academy, 21 Calea Grivitei, 010702 Bucharest, Romania}
\email{andrei.moroianu@math.cnrs.fr}

\thanks{This work was partly supported by the PNRR-III-C9-2023-I8 grant CF 149/31.07.2023 {\em Conformal Aspects of Geometry and Dynamics}.}

\begin{document}

\begin{abstract} Projective structures on curves appear naturally in many areas of mathematics, from extrinsic conformal geometry to analysis, where the main problem is to find qualitative information about the solutions of Hill equations. In this paper, we describe in detail the correspondence between different equivalent definitions of projective structures and their isomorphism classes, correcting long-standing inexactitudes in the literature. As an application, we show that the {\em Yamabe problem for curves} in a conformal/Möbius ambient space has no solutions in general. 
  
 \bigskip


\end{abstract}

\maketitle

\section{Introduction}

There are not many structures one can define on a 1-dimensional manifold, and all we can think of are {\em flat}, i.e. they do not admit local invariants.

For example, a conformal structure on a curve is simply its smooth structure, and all curves are orientable, even parallelizable. A Riemannian structure on a connected curve is characterized globally by its total length, and, for the ones of infinite length, whether they are complete or not.

There is, however, a quite interesting type of structures on a curve, and the moduli space of such structures -- in case of closed curves -- turns out to be really intricate. We are talking here about {\em projective structures}.

Historically, projective structures have been studied for a long time as particular cases of a $(G,X)$-geometry (a viewpoint first introduced by Felix Klein in his Erlangen Program, then formalized by Thurston). For the $1$-dimensional case, Kuiper gave a classification in 1954 \cite{ku}, but there are some flaws as we will see in the present paper.

Projective structures on curves are tightly linked with the vast domain of analysis focusing on {\em Hill's equation}
$$x''+Fx=0,\ F:\R\ra\R,\ F(t+T_0)=F(t), \ \forall t\in\R.$$
Indeed, two independent solutions $(x_1,x_2)$ of the above equation define a map
$$\R\ni t\mapsto [x_1(t):x_2(t)]\in\rp^1$$
that provides local charts forming a projective atlas as a $(\PSL(2,\R),\rp^1)$-geometry, \cite{mw}.

The latter analytic approach can be stated in an invariant way as a {\em Laplace structure} on a curve $C$, which is a second-order linear differential operator, with no first-order term (there is indeed an invariant way to state this, see Proposition \ref{difhes}) on a certain {\em weight bundle} on $C$ (cf. Definition \ref{lapla}).

Our interest is motivated by the existence of various structures induced on a curve $C$ embedded in a conformal manifold $M$ by the conformal structure of the ambient space (more precisely, the ambient space $M$ needs to be {\em conformal} if $\dim M\ge 3$, and furthermore {\em Möbius} if $\dim M=2$, see Section 2 and \cite{mlc}). Every such curve carries special parametrizations, for which the {\em conformal acceleration} vanishes, and a canonically induced {\em Laplace structure}.
We show in Proposition \ref{specproj} that these notions are strongly related.

The induced Laplace structure on a curve in a conformal/Möbius ambient space is, in fact, a particular case of the {\em conformal Laplacian}, or {\em Yamabe operator}, that is, a second order linear differential operator on a weight bundle, with symbol equal to the conformal contraction and  with vanishing first order term, determined by the conformal structure alone. For all dimensions $n\ge 3$, the {\em Yamabe problem} consists in searching for a metric with constant scalar curvature in the conformal class (which turns out to be, up to a constant factor, the zero order term of the above Laplace structure, viewed in an appropriate gauge). 

In the 1-dimensional case, Calderbank and Burstall formulated in \cite{cal-bur} the {\em Yamabe problem for curves} in a conformal/Möbius ambient space: is there a metric in the conformal class for which the induced Laplace structure has constant zero order term?

In order to give a (negative) answer to this question, we classify the projective structures on connected curves up to isomorphism. This endeavor is not new, we could find many references dealing with this topic, starting with the classification of these structures by Kuiper \cite{ku}, the classification (up to diffeomorphism) of vector fields preserving a projective structure on a circle \cite{hit}, and many other authors interested in the qualitative behavior of solutions of Hill's equation \cite{guh}. 

Nevertheless, we consider that our differential-geometric approach brings more clarity and situates the study of projective in a broader context of differential geometry, like Lie group actions and differential operators associated to a given $G$-structure (in the sense of Kobayashi \cite{kob}). 

We prove several versions of the equivalence between a Laplace structure on a curve and a projective structure, which in turn can be characterized by a {\em developing map} and its {\em holonomy}, \cite{gold}, see Theorem \ref{main}. 

The developing map is the main tool to characterize a projective structure on an open curve (diffeomorphic to $\R$), cf. Theorem \ref{open}, while the holonomy turns out to determine the isomorphism class of a projective structure on a closed curve (diffeomorphic to $S^1$), cf. Theorem \ref{clas-g}.

For open curves, the developing map $D:C\ra\tp$ defines an embedding of the curve in the universal cover of $\rp^1$, and we define the {\em winding number} of the curve as the "average" number of points in the preimage in $C$ of the points of $\rp^1$ (if this "average"is not integer, we "round it up" to a number in $\tfrac12+\N$). In  Proposition \ref{open} and in Corollary \ref{openn} we retrieve the classification of open projective curves, first obtained in \cite{ku}, see also \cite{gold}.

Using some technical results about $\tsl$, we approach in Section 5 the classification of projective structures on closed curves, in which the main role is played by the conjugacy class of a generator of the {\em holonomy group}, which acts freely on the image of the developing map. In fact, there exists a canonical generator (for a given orientation on the curve) of the holonomy group, and its conjugacy class determines the projective structure on the curve up to isomorphism, Proposition \ref{dev-sl+} and Corollary \ref{moduli}.\

We complete the theory by Theorem \ref{clasif}, which produces a natural, continuous and bijective map from the set of Laplace structures on $S^1$, up to diffeomorphism, and the set of conjugacy classes of {\em positive generators} of the holonomy groups of projective structures on $S^1$. We use here the classification of conjugacy classes of $\tsl$, Theorem \ref{conjtsl}, which seems to be a classical result but was only implicitly (and inexactly) used in \cite{ku}, however it is only very recently that a complete study of the conjugacy classes of elements of $\tsl$ has been published \cite{taf}.

Finally, we describe in Proposition \ref{afin}, and Proposition \ref{windnumbers} some geometric properties that characterize each projective class, and in Proposition \ref{autom} we determine the automorphism group of each projective class. For the impatient reader, the table at the end of Section 5 provides most of the relevant information.

We conclude the paper by the proof of the fact that the zero order term of a Laplace structure can not always be constant, giving therefore a negative answer to the Yamabe problem for curves in a conformal/Möbius ambient space, cf. Proposition \ref{yam}.

  \section{Curves in conformal geometry} 
  \begin{ede} On a $n$-dimensional manifold $M$, denote by $L$ the (oriented) line bundle associated to the representation
    $$\GL(n,\R)\ni A\mapsto |\det A|^{1/n}\in\R^*_+$$
  and we refer to it as the {\em weight bundle of weight 1}.\end{ede}
Analogously, using other powers of the absolute value of the determinant, one can consider any real power $L^k$ of the weight bundle $L$. We can easily notice that $L^0$ is canonically isomorphic to the trivial line bundle $M\times\R$, $L^{-1}$ is isomorphic with the dual of $L$, and $L^k$, with $k\in N$, is isomorphic to the $k$-th tensorial power $L^{\otimes k}$ of $L$.
Clearly, $L^k$ are all oriented (we know what a {\em positive} section of these bundles is), but there is no canonical trivialization. However, a trivialization of $L^k$, for $k\ne 0$, induces a trivialization of all nonzero powers of $L$, including $L^{-n}\simeq |\Lambda M|$, the bundle of {\em densities} on $M$.

More generally, a tensor bundle on $M$ (i.e., a bundle associated to the the frame bundle of $M$ via a representation $\rho:\GL(n,\R)\ra \GL(m,\R)$)  is said to have {\em intrinsic  weight} $k\in \R$ if and only if $\rho(a\mbox{Id}_n)=a^k\mbox{Id}_m$ for every $\a\in \R^*_+$. The bundles $L^k$ make no exception to this rule, and $k$ is indeed the {\em intrinsic weight} of $L^k$ for every $k\in \R$.
    
\begin{ede} A {\em conformal structure} on a manifold $M$ is a symmetric, non-degenerate bilinear form $c$ on $\T M$ with values in $L^2$, which is normalized, in the sense that the induced bilinear form on $|\Lambda^nM|\simeq L^{-n}$ with values in $L^{-2n}$ is the identity of the square of the latter bundle.
\end{ede}
In this paper we will only discuss {\em Riemannian} conformal structures, i.e. such that $c$ is {\em positive definite}, which is a notion that makes sense because the target space of $c$ is $L^2$, a canonically oriented bundle. 

Notice that all bundle homomorphisms induced by $c$ (including $c$ itself, as a bundle map from $TM\otimes TM$ to $L^2$) are between spaces with the same intrinsic weight. In particular, the so-called {\em rising/lowering of indices} from $\T^*M$ to $\T M$ must be adjusted to satisfy this ``weight conservation'' principle. Thus, $c$ defines an isomorphism between $\T^*M$ and $\T M\otimes L^{-2}$  

The weight bundles $L^k$, for $k\in\R$, are also important in the case of a $1$-dimensional manifold, a curve. Here, of course, a conformal structure $c$ is nothing else but the very smooth structure of the curve.

For further notions and results about conformal geometry, we refer to \cite{mlc}. Let us briefly summarize some of them:
\begin{ede} A {\em Weyl structure} is a torsion-free conformal connection on $(M,c)$. It induces connections on all bundles associated to the frame bundle of $M$.\end{ede}
Two Weyl structures $\nb$ and $\nb'$ are related by 
\be\label{Weylrel} \nb'_XY-\nb_XY=\tilde\theta_XY:=\theta(X)Y+\theta(Y)X-c(X,Y)\theta,
\ \mbox{for}\ X,Y\in\T M,\ee
where $\theta$ is a 1-form. Note that $c(X,Y)\theta\in L^2\otimes \T^*M\simeq \T M$.

A remark (for example, from \cite{mlc}) is that for every curve in a conformal manifold, the latter carries {\em adapted} Weyl structures (associated to parametrizations $\gamma:I\ra C\subset M$ of the curve), for which the parametrized curve is {\em geodesic}, i.e. $\nabla_{\dot\gamma}\dot\gamma=0$. In order to single out a particular class of parametrizations or even a class of immersed curves, one needs to introduce further objects induced by $c$.

The {\em Schouten-Weyl tensor} $h^\nabla$ of a Weyl structure $\nabla$ is a bilinear form on $\T M$ and it is computed, for $n\ge 3$, from the curvature of $\nabla$ \cite{ga}, \cite{pg}, see also \cite[Definition 2.9]{mlc}. The Schouten-Weyl tensor satisfies the following transformation rule when the Weyl structures are changed via \eqref{Weylrel}, as shown in \cite{ga}, \cite{pg}:
\be\label{shout} h^{\nb'}=h^\nb -\nb\theta+\theta\otimes\theta-\frac12c(\theta,\theta)c.\ee
The formula from \cite[Definition 2.9]{mlc} cannot be used for $n=2$, but one can define a Schouten-Weyl-like tensor on a {\em Möbius} surface \cite{cald}, see also \cite[Definition 3.2]{mlc} (which is a structure that "rigidifies" the conformal structure of a 2-dimensional conformal manifold), and this tensor satisfies the gauge transformation law \eqref{shout} as well. We use here the notion of a Möbius structure merely to be able to have a 2-dimensional ambient space and still get the same results as for higher-dimensional conformal spaces. 

We make the following convention of naming a {\em conformal/Möbius} manifold $(M,c)$: if $\dim M\ge 3$ this will simply mean $(M,c)$ as a conformal manifold, if $n=2$ we will assume $M$ is endowed with a Möbius structure that is compatible with $c$. The consequence in both cases is that, for every Weyl structure $\nb$, there is an associated Schouten-Weyl tensor $h^\nb$, which changes with the Weyl structure according to \eqref{shout}.

We are interested in the structure induced by a conformal/Möbius ambient space on a curve.

Let us first recall the definition of the {\em conformal acceleration} of a parametrized curve in conformal/Möbius ambient space \cite[Definition 4.1 and Lemma 4.8]{mlc}:
\begin{ede} \cite{mlc} The {\em conformal acceleration} $a_c(\c)$ of a curve $\c:I\ra\M$ is equal to the vector field $c(\dot\c,\dot\c)h^\nb(\dot\c)$ where $\nb$ is any Weyl structure adapted to $\c$. \end{ede}
As before, the factor $c(\dot\c,\dot\c)\in L^2$ ensures that the conformal acceleration is a vector field on $C:=\c(I)$. This vector field has a {\em normal} and a {\em tangential} component 
$$a_c(\c)=a_c^N(\c)+a_c^T(\c),$$
where $a_c^N\perp\dot\c$ and $a_c^T(c)\in \T C$ are both independent of the choice of the Weyl structure $\nb$, and $a_c^N(\c)$ is furthermore independent also of the parametrization of $C$ \cite[Lemma 4.2]{mlc}. 

We define the following class of parametrizations of any embedded curve $C\subset (M,c)$ \cite{baea}:
\begin{ede}\label{speccf} A (local) parametrization of a curve $C$  in a conformal/Möbius manifold is {\em special} if and only if the tangential part of its conformal acceleration vanishes identically.
\end{ede}
Such parametrizations always exist, as solutions of an ODE. An equivalent description of special parametrizations will be given in Subsection 3.2, Proposition \ref{specproj}.

\section{Laplace structures on curves, projective parametrizations and
   Hill's equation}\label{hill} 

 We review here two main descriptions of projective structures on  curves and describe the equivalence of these approaches. More precisely, a projective structure on a curve is defined by:
 \begin{ede}\label{defproj} A {\em projective structure} on a given curve $C$
  is an
  atlas $\mathcal{A}$ of charts from open subsets of $C$ with values
  in $\R$, such that the transition maps consist of restrictions of maps of the form
$$f_A(x):=\frac{ax+b}{cx+d}, \mbox{ for }
    A:=\begin{pmatrix} a&b\\c&d\end{pmatrix}\in \GL(2,\R),$$
called {\em homographies}. A {\em projective curve} is a curve $C$ 
together with such an atlas $\mathcal{A}$, and a {\em projective map}
$f:C_1\ra C_2$ between projective curves is a map that is induced, in
corresponding projective charts of $\mathcal{A}_1$, respectively
$\mathcal{A}_2$, by homographies.
\end{ede}
A projective map is thus a local diffeomorphism. Conversely, if $f$ is
an immersion $f:C_1\ra C_2$ of curves, and $C_2$ has a projective
structure, then $f$ induces a projective structure on $C_1$.

A classical result states that the homographies $f_A$, as described in Definition \ref{defproj}, satisfy the equation $S(f_A)=0$, where $S$ is the {\em Schwarzian derivative}:
\be\label{schwarz}S(\f):=\frac{\f'''}{\f'}-\frac32\left(\frac{\f''}{\f'}\right)^2,\ee
defined on local diffeomorphisms $\f:I_1\ra I_2$ between intervals in $\R$.

Conversely, for every local diffeomorphism $\varphi$ satisfying $S(\varphi)=0$ there exists a matrix $A$ such that $\varphi$ coincides on an open set with $f_A$.

In the next section we will investigate in full detail the projective structures on curves from the viewpoint of $(G,X)$-structures (in our case, $(\PGL(2,\R),\rp^1)$-structures), but in the present section we focus on a class of differential-geometric structures on a curve, the {\em Laplace structures}, that we will further show to be equivalent to the projective structures from Definition \ref{defproj}.

\subsection{Laplace structures on curves}

Let $C$ be a smooth curve. A connection $\nb$
on $\T C$ induces a connection on the canonical weight bundle
$L:=|\T C|$ and on every real power of $L$. 

\begin{remark}\label{41} $L=|\T C|$ is the oriented line bundle associated to $\T C$;
  every orientation of $C$ induces a canonical isomorphism of $L$ with
  $\T C$. The bundle $L$ being oriented (hence topologically trivial), all real
  powers of $L$ are well defined and trivial. In the sequel, if $X$ is a non-vanishing vector field along $C$, e.g., the speed vector field $\dg$ induced by some parametrization $\c$, then
  $|X|^k$ will denote the corresponding section of $L^k$.
\end{remark}

Every parametrization $\gamma$ of $C$ induces a unique
connection $\nabla^\gamma$ defined by requiring $\dg$ to be $\nabla^\gamma$-parallel. Conversely, every connection admits
compatible {\em local} (but not necessarily global) parametrizations. For example,
the tautological connection on the half-line $(0,+\infty)$ induces a connection on the closed curve $C:=(0,+\infty)/\sim$, where $t\sim
2t, \ \forall t>0$, which carries no global parallel sections.

If $\nabla$ is any connection on $C$, the corresponding connection on $L^k$ defines a {\em Hessian} $\Hss^\nb:\ci(L^k)\to \ci(\T ^*C\otimes \T ^*C\otimes L^k)\simeq \ci(L^{k-2})$, acting on a section $l$ on $L^k$ by
$$\Hss^\nb l(X,Y):=\nb_X\nb_Yl-\nb_{\nb_XY}l.$$
When comparing the Hessians induced by two
connections on $L^k$, \cite[Proposition 3.7]{mlc} (see also \cite{hit}) shows:

\begin{elem}\label{difhes} The difference between the Hessians on $L^k$ of two
  arbitrary connections $\nb,\nb'$ on a curve $C$ is a zero-order 
operator if and only if $k=1/2$.
\end{elem}
\begin{proof} The difference $\nb'-\nb$ is a 1--form $\theta\in\ci(\T ^*C)$. For every  section $l$ of $L^k$ one thus has $\nb'_Xl-\nb_Xl=k\theta(X)l$. As the
  Hessians $\Hss^{\nb'} l$ and $\Hss^\nb l$ are bilinear forms on $\T C$ with values in $L^k$,
  we can choose the arguments $X,Y\in \T C$ to be non-zero, and such that $\nb_XY=0$
  at the point where we make the computation. This implies
\bea(\Hss^{\nb'}l)(X,Y)&=&\nb'_X\nb'_Yl-\nb'_{\nb'_XY}l=
\nb'_X(\nb_Yl+k\theta(Y)l)-\nb'_{(\theta(X)Y)}l\\
&=&\nb_X\nb_Yl+k(\nb_X\theta)(Y)l + k\theta(Y)\nb_Xl + 
k\theta(X)\nb_Yl \\
&& +k^2\theta(X)\theta(Y)l-\theta(X)\nb_Yl-k\theta(X)\theta(Y)l\\ 
&=&(\Hss^\nb l)(X,Y)+(2k-1)\theta(X)\nb_Yl\\
&&+k\left(\nb_X\theta(Y)+(k-1)\theta(X)\theta(Y)\right)l,\eea
where in the last equality we use the fact that $\theta(X)\nb_Yl=\theta(Y)\nb_Xl$ since $X$ and $Y$ are collinear. This concludes the proof.
\end{proof}
The reason for which we included here the proof instead of merely citing \cite{mlc}, is to obtain the explicit relation between Hessians on $L^{1/2}$:
\be\label{comph1}\Hss^{\nb'}-\Hss^\nb =\tfrac12\left(\nabla\theta-\tfrac12\theta\otimes\theta\right).\ee
If $\nabla^{\gamma_1}$ and $\nabla^{\gamma_2}$ are the connections induced by two parametrizations $\gamma_1$, $\gamma_2$ of $C$, then the above formula reads
\be\label{comph}\Hss^{\gamma_2}-\Hss^{\gamma_1}=\tfrac12\left(\nabla^{\gamma_1}\theta-\tfrac12\theta\otimes\theta\right),
\ee
where $\theta:=\nabla^{\gamma_2}-\nabla^{\gamma_1}\in\ci(\T ^*C)$.

\begin{ede}\label{lapla} (See also \cite[Definition 3.8]{mlc}) A {\em Laplace structure} $\mathcal{L}$ on a smooth curve $C$ is a second order
  linear differential operator on $L^{1/2}$ with values in $\T ^*C\otimes \T ^*C\otimes L^{1/2}\simeq L^{-3/2}$, which differs from some
  (hence any) Hessian $\Hss^\nabla$ by a zero-order term.
\end{ede}

Note that some authors, (e.g. Hitchin \cite{hit} or Guha \cite{guh}) use the term {\em projective connection} to denote what we call here a Laplace structure. Their main reason is that, although the operator has second order, it shares some features with linear connections: its symbol, as a bundle map from $\Sym^2(\T ^*C)\otimes L^{1/2}\simeq L^{-3/2}$ 
to $L^{-3/2}$, is the identity, and (precisely because it contains no first-order derivatives), it also satisfies a relation similar to the Leibnitz rule. However, we prefer the terminology {\em Laplace structure} since it refers to a second-order operator (which moreover can also be defined on higher-dimensional manifolds \cite{mlc}).

Recall (see Remark \ref{41}) that any (local) parametrization $\c:I\ra C$ of $C$ defines privileged local sections $|\dg|^k\circ\c^{-1}\in\ci(L^k)$ for every $k$. A Laplace structure $\LL$ on $C$ thus induces, for each parametrization $\c$, a second order operator 
$\mathcal{L}^\gamma$ acting on functions $x:I\to\R$ as follows:
\be\label{lapl1}|\dg|^{-3/2}\mathcal{L}^\c(x)\circ\c^{-1}:=\mathcal{L}((x\circ\c^{-1})|\dg|^{1/2})\ee
(this is an equality between sections of $\LL^{-3/2}$ defined along the open set $\gamma(I)\subset C$).
Lemma \ref{difhes} shows that the operator
 $\mathcal{L}^\c$ differs from the
second derivative on functions by a zero-order term (i.e., it does not involve any first-order derivative):
\be\label{eqh}\mathcal{L}^\c x=x''+F^\c x.\ee

\begin{prop}\label{la-hi} A Laplace structure $\mathcal{L}$ on a curve $C$ determines a map
$$\mathcal{F}:\{\mbox{local parametrizations $\c:I\ra C$}\}
\lra\{\mbox{functions on   $I$}\},\qquad \gamma\mapsto F^\gamma.$$
such that the dependence of $F^{\textstyle{\c}}$ on $\c$ satisfies the following rule with respect to gauge transformations:
If $\c_1:I_1\to C$ is a local parametrization and $\f:I_2\to I_1$ is a diffeomorphism, then the function $F^{\textstyle{\c}_{\scriptscriptstyle{2}}}$ associated to the local parametrization $\c_2:=\c_1\circ\f$ reads
\be
\label{0ord}
F^{{\textstyle{\c}}_{\scriptscriptstyle{2}}}=(\f')^2F^{{\textstyle{\c}}_{\scriptscriptstyle{1}}}\circ\f+\tfrac12S(\f),
\ee
where $S(\f)$ is the Schwarzian derivative of $\f$ \eqref{schwarz}.
\end{prop}
\begin{proof} Let $\LL$ be a Laplace structure on $C$, let $\c_1$ and $\c_2:=\c_1\circ\f$ be two local parametrizations of $C$ and let $x_1:I_1\to\R$ be any smooth function. We denote by $x_2:=x_1\circ\f$. From \eqref{lapl1} and \eqref{eqh} it follows that
\be\label{7}
x_1''+F^{\c_1} x_1=(|\dg_1|^{3/2}\mathcal{L}((x_1\circ\c_1^{-1})|\dg_1|^{1/2}))\circ\c_1
\ee
and 
\be\label{8} x_2''+F^{\c_2} x_2=(|\dg_2|^{3/2}\mathcal{L}((x_2\circ\c_2^{-1})|\dg_2|^{1/2}))\circ\c_2.
\ee
On the other hand, the local vector fields $\dg_1$ and $\dg_2$ along $C$ are related by $\dg_2=\f'\circ\c_2^{-1}\dg_1.$ Using that $x_1\circ\c_1^{-1}=x_2\circ\c_2^{-1}$, we can write the argument of the right hand side of \eqref{7} as $(x_1\circ\c_1^{-1})|\dg_1|^{1/2}=(x_2(\f')^{-1/2})\circ\c_2^{-1})|\dg_2|^{1/2}$.
Applying \eqref{8} to $(x_2(\f')^{-1/2}$ instead of $x_2$ yields 
\be\label{9} (x_2(\f')^{-1/2})''+F^{\c_2}(x_2(\f')^{-1/2})=(|\dg_2|^{3/2}\mathcal{L}((x_1\circ\c_1^{-1})|\dg_1|^{1/2}))\circ\c_2.
\ee
Comparing \eqref{7} and \eqref{9} we thus obtain
\be\label{10} (x_2(\f')^{-1/2})''+F^{\c_2}(x_2(\f')^{-1/2})=(\f')^{3/2}(x_1''+F^{\c_1} x_1)\circ\f.
\ee
Since $x_2=x_1\circ\f$, we have $x_2'=\f'x_1'\circ\f$ and $x_2''=(\f')^2x_1''\circ\f+\f''x_1'\circ\f$, whence
\bea(x_2(\f')^{-1/2})''&=&x_2''(\f')^{-1/2}+2x_2'((\f')^{-1/2})'+x_2((\f')^{-1/2})''\\
&=&(\f')^{3/2}x_1''\circ\f+(\f')^{-1/2}\f''x_1'\circ\f-(\f')^{-1/2}\f''x_1'\circ\f+x_2((\f')^{-1/2})''\\
&=&(\f')^{3/2}x_1''\circ\f+x_2((\f')^{-1/2})''.
\eea
From \eqref{10} we thus get
$$(\f')^{-1/2}F^{\c_2}+((\f')^{-1/2})''=(\f')^{3/2}F^{\c_1}\circ\f,$$
which is equivalent to \eqref{0ord}.
\end{proof}

Conversely, every such map $\mathcal{F}$ satisfying the above gauge transformation rule defines a Laplace structure on $C$. Indeed, for every local parametrization $\c:I\to C$, the function $F^\c$ on $I$ defines by \eqref{eqh} a second order differential operator $\LL^\c$ on functions, which in turn defines a second order differential operator $\LL$ on $L^{1/2}$ (along the image of $\c$). The fact that this operator does not depend on the choice of $\gamma$ follows, by reversing the previous argument, from the transformation rule \eqref{0ord}.

We state now the main result of this section:
\begin{thrm}\label{main}
Let $C$ be a closed curve. Then there is a
canonical equivalence between:
\bi
\item Projective structures on $C$;
\item Laplace structures $\LL:\ci(L^{1/2})\ra \ci(L^{-3/2})$.
\ei
Moreover, if $\mathcal{A}$ is the projective atlas corresponding to $\LL$, we have:
\bi
\item[(a)] For any local parametrization $\c:I\to C$, the induced operator
  $\mathcal{L}^\c$ acting on functions on $I$ is the second derivative (i.e., ${F}^\c\equiv 0$) 
if and only if $\c$ is locally the inverse of a projective chart.
\item[(b)] For any choice of an orientation of $C$ (and thus of an isomorphism $L\simeq \T C$) a local non-vanishing section $l$ 
of $L^{1/2}$ satisfies the Laplace equation
\be\label{lll}\mathcal{L}l=0,\ee 
if and only if there exists a projective chart compatible with the orientation whose inverse $\gamma$ satisfies $\dg=l^2$.
\item[(c)] For any pair of linearly independent solutions
$l_1,l_2$ of Equation \eqref{lll}, the map
\be\label{l12}l_{1,2}:C\ra\R^*,\ l_{1,2}(p):=\frac{l_1(p)}{l_2(p)},\ee
defined where $l_2\ne 0$, is a projective chart.
\ei
\end{thrm}
\begin{proof}
Let $\mathcal{A}$ be a projective structure on $C$ as in Definition \ref{defproj}. The charts of $\mathcal{A}$ will be referred to as {\em projective charts}. The inverse, $\c_1:I_1 \to \c_1(I_1)\subset C$ of a projective chart is a local parametrization of $C$.
We define a Laplace structure along $\c_1(I_1)$ by 
$$\mathcal{L}_1:=\Hss^{\c_1}\mbox{ on sections of } L^{1/2}.$$
This is equivalent to say that the function $F^{\c_1}$ defined by $\mathcal{L}_1$ in Proposition \ref{la-hi} vanishes. If $\c_2:I_2\to C$ is the inverse of another projective chart, one can write on $\c_2^{-1}(\c_1(I_1))$ the relation $\c_2=\c_1\circ\f$ for some homography $\f:\c_2^{-1}(\c_1(I_1))\to\c_1^{-1}(\c_2(I_2))$.
Because of \eqref{0ord} and by the
well-known fact that homographies have vanishing Schwarzian
derivative, it turns out that $\mathcal{L}_1$ and $\mathcal{L}_2$ coincide on their common domain of definition. Consequently, the local Laplace structures defined in this way glue together to define a global 
Laplace structure $\LL^\mathcal{A}$ on $C$.

Conversely, suppose $\mathcal{L}$ is a Laplace structure on $C$. We claim that around any point 
there exist local parametrizations $\c$ such that $F^\c=0$. Indeed, pick any local
parametrization $\c_0$ of $C$ and look for a re-parametrization
$\c:=\c_0\circ\f$ such that $F^{\c}=0$.
From \eqref{0ord}, this is a third order non-linear differential equation for $\f$:
$$S(\f)+2(\f')^2F^{\c_0}\circ\f=0,$$
which obviously has local solutions $\f$ around each point. 

If $\c_1$ and $\c_2$ are local
parametrizations with ${F}^{\c_1}={F}^{\c_2}=0$, and $p$ is any point in $\mathrm{Im}(\c_1)\cap\mathrm{Im}(\c_2)$, then 
the local inverses of $\c_1$ and $\c_2$ near $p$ are real-valued charts on $C$ whose coordinate change $\f:= \c_1^{-1}\circ\c_2$ 
is a homography because by \eqref{0ord} it satisfies the equation $S(\f)=0.$ The set of charts constructed in this way defines a projective atlas $\mathcal{A}^\LL$. The correspondences $\mathcal{A}\mapsto \LL^\mathcal{A}$ and $\LL\mapsto \mathcal{A}^\LL$ are clearly inverse to each other.

We now turn to the second part of the statement. The point (a) follows directly from \eqref{eqh}. 
To prove (b), note that every non-vanishing section $l$ of
$L^{1/2}$, together with the choice of an orientation of $C$ (inducing an isomorphism of $\T C$ with $L$) defines local parametrizations $\c$ (unique up to composition with a translation) such that $\dg=l^2$. 
If $\LL l=0$ then \eqref{lapl1} and \eqref{eqh} applied to the constant function $x\equiv 1$ yield $F^\c\equiv 0$, so $\c$ is by (a) the inverse of a projective chart. Conversely, if $\c$ is the inverse of a projective chart then $F^\c=0$ and from \eqref{lapl1} and \eqref{eqh} it follows that $l:=|\dg|^{1/2}$ satisfies $\LL l=0$.

For (c), we choose a local projective parametrization (i.e. inverse of projective chart) $\gamma:I\to C$. By (a), we have $F^\c=0$. We write (for $i=1,2$) $l_i=(x_i\circ\c^{-1})|\dg|^{1/2}$, where $x_i:I\to\R$ are smooth functions. Then by \eqref{lapl1} and \eqref{eqh} we get $x_i''=0$, so $f:=x_1/x_2$ is a homography. Consequently $l_1/l_2=f\circ\gamma^{-1}$ is the composition of a homography with a projective chart, i.e. it is a projective chart as well.
\end{proof}
\begin{remark}\label{iso-lap-proj} The map $C\ni x \mapsto l_1(x)/l_2(x)\in\R$ from the point (c) in Theorem \ref{main} is defined only on the open set where the solution $l_2$ of the Laplace equation does not vanish. Instead we can consider, for any pair $l_1,l_2$ of linearly independent solutions of \eqref{lll} (note that they are defined globally on $\tilde C\simeq\R$ as solutions to a linear ODE), the map
\be\label{proj-pair}[l_1:l_2]:\tilde C\ra \rp^1,\ x\mapsto [l_1(p(x):l_2(p(x)],\ee
which is well-defined because the isomorphisms of $L^k\oplus L^k$ with $\R^2$ induced by any trivialization of $L$ all induce the same identification between $\mathbb{P}(L^k\oplus L^k)$ and $\mathbb{P}(\R^2)=\rp^1$. As shown in Theorem \ref{main}, the map $[l_1:l_2]$ is locally a projective chart, hence any lift:
\be\label{dev-lap} D_{l_1,l_2}:\tilde C\ra\tp\ee
 to $\tp$ of this map
is a {\em developing map} for the induced projective structure -- see Definition \ref{devhol} below.\end{remark}

\subsection{Laplace structures on curves in conformal/Möbius geometry}

A {\em Laplace structure} \cite[Definition 3.8]{mlc} on a conformal manifold (of any dimension $n\ge 1$) is defined as a second-order linear differential operator, with vanishing first order term and symbol equal to the conformal contraction ($c:\T^*M\ot\T^*M\ra\L^{-2})$ on sections of $L^{1-n/2}$, cf. \cite{mlc}. For $n=1$, we get Definition \ref{lapla} above. 

Moreover, if $M$ is conformal manifold of dimension $n\ge 3$ or a Möbius surface, then any curve $C\subset M$ turns out to inherit a canonical Laplace structure induced by the ambient space \cite[Proposition 4.25]{mlc}:
\begin{prop}\label{lapl-ind} For a curve $C$ in a conformal/Möbius manifold $(M,c)$, the following operator is a Laplace structure which only depends on the embedding and the structure of the ambient space:
\be\label{lapcf} \mathcal{L}^c:\ci(L^{1/2})\ra \ci(L^{-3/2}),\ \mathcal{L}^cl:=\Hss^\nb l+\frac12 \frac{h^\nb(\dot\c,\dot\c)}{c(\dot\c,\dot\c))},\ee
where $\c$ is a local parametrization and $\nb$ an adapted Weyl structure.\end{prop}
As a consequence of Theorem \ref{main}, any curve in a conformal/Möbius ambient space inherits a projective structure, and the local projective parametrizations $\c$ for this structure are such that $h^\nb(\dot\c,\dot\c)=0$ for the corresponding adapted Weyl structure. We can therefore make the connection with the {\em special parametrizations} defined in Definition \ref{speccf}:
\begin{prop}\label{specproj} For a curve in a conformal/Möbius manifold, the {\em special} parametrizations (for which the tangential conformal acceleration vanishes) are exactly the projective parametrizations for the induced Laplace structure.\end{prop}
Therefore, the projective structure on a curve in a conformal/Möbius ambient space can be either seen via the projective charts from Theorem \ref{main} (c) (for the induced Laplace structure) or via the parametrizations that have vanishing  tangential part of the conformal acceleration.

\section{Projective structures on curves}\label{psc}

In this section we discuss the projective structures on open and closed curves from the classical viewpoint of Felix Klein's Erlangen Program, and present a complete classification of these structures up to isomorphism, see also
\cite{gold}, \cite{hit}, \cite{ku}.

 A homography $f_A$ as in Definition \ref{defproj} actually only depends on the image $\hat A\in \PGL(2,\R)$ of the matrix $A$, so we may also think of a projective structure as of a $(G,X)$-structure in the sense of Thurston (see also the works of Kuiper, Kobayashi, Ehresmann, Klein), where $G:=\PGL(2,\R)$ and $X=\rp^1$. We recall:
  \begin{ede} Let $G$ be a Lie group acting effectively and transitively on the manifold $X$. A {\em $(G,X)$-structure} on a manifold $M$ is an atlas of diffeomorphisms $\varphi:U\ra\varphi(U)\subset X$ such that, for any two charts $\varphi_1,\varphi_2$, on each connected component of its domain of definition, the transition function $\varphi_2\circ\varphi_1^{-1}$ coincides with the action on that connected component of an element of $G$.

  A map $f:M\ra N$ between two $(G,X)$-manifolds is called a {\em $G$-map} if, for every charts $\varphi:U\ra \varphi(U)\subset X$, $\psi:V\ra \psi(V)\subset X$ on $M$, resp. $N$, with $U$ connected and small enough such that $f(U)\subset V$, the composition $\psi\circ f\circ \varphi^{-1}$ is given by the action of an element $g\in G$.\end{ede}

If $G$ is connected and if $\tl G$ acts effectively on $\tl X$, it is equally possible to consider that a $(G,X)$-structure induces a $(\tl G,\tl X)$-structure and conversely, where $\tl G$, $\tl X$ denote the universal covers of $G$, resp. $X$. This viewpoint is particularly useful when defining the {\em developing map} and its {\em holonomy}  (see below). As all curves are orientable, there is no loss of generality to first consider {\em oriented} projective structures on $C$ compatible with a given orientation (i.e., an atlas of charts sch that the transition functions are orientation-preserving homographies of type $f_A$ with $A\in\SL(2,\R)$) and then study orientation-reversing projective isomorphisms.

An {\em oriented projective} structure on $C$ is thus a $(\PSL(2,\R),\rp^1)$-structure or, equivalently, a $(\tsl,\tp)$-structure.  The advantage of this latter description is that both the model space and the structure group are simply connected; its drawback is that $\tsl$ is not a linear group.

\subsection{The developing map and holonomy of a closed curve}

The general theory of $(G,X)$-structures \cite{gold} on a manifold $M$ states that such a structure is equivalent with a pair $(D,\rho)$ of a {\em developing map} $D:\tl M\ra X$ which is $\rho$-equivariant, $\rho:\pi_1(M)\ra G$ being the {\em holonomy} homomorphism, up to equivalence: 
\begin{ath}\cite{gold} Let $M$ be a $(G,X)$-manifold. Then there exists a $G$-map $D:\tl M\ra X$, and a group homomorphism $\rho:\pi_1(M)\ra G$ such that 
  \be\label{deveq}\tl D(\c x)=\tl\rho(\c)\tl D(x),\ \forall x\in\tl M,\ \forall \c\in\pi_1(M).\ee
  The set of pairs $(D,\rho)$ as above has an equivalence relation:
  \be\label{eq-dev}(D,\rho)\sim(D',\rho') \Longleftrightarrow \exists g\in G\ :\ D'(.)=g D(.)\mbox{ and }\rho'=g\rho g^{-1},\ee
  and the equivalence classes are in 1--1 correspondence to the isomorphism classes of $(G,X)$-structures on $M$.
\end{ath}
\begin{ede}\label{devhol} 
  $D$ is called a {\em developing map} and $\rho$ its {\em holonomy}.\end{ede}
Instead of a full proof for the classical result above (for which we refer to \cite{gold}), we recall here briefly the way the holonomy homomorphism appears in the general theory of $(G,X)$-structures:

 A developing map is obtained by gluing $(G,X)$-charts on $M$, using elements of $G$ on the intersections, and this gluing turns out to yield global maps from $\tl M$ to $X$. The fundamental group $\pi_1(M)$ acts on $\tl M$ by $(G,X)$-maps, which necessarily correspond to elements of $G$, and that defines the group homomorphism $\rho$.

If $\tl G$ acts effectively on $\tl X$, it is equally possible to consider that a $(G,X)$-structure induces a $(\tl G,\tl X)$-structure and conversely. In our case, we can consider an {\em oriented projective structure} on $C$ to be a $(\tsl,\tp)$-structure, whose isomorphism class is in turn characterized by a pair $(D,\rho)$, where $D:\tl C\ra\tp$ is the developing map and $\rho:\pi_1(M)\ra \tsl$ is its holonomy, up to the equivalence relation \eqref{eq-dev}.

Equivalently, an oriented projective structure is defined by an atlas  like in Definition \ref{defproj} but where only homographies $f_A$ with $A\in\SL(2,\R)$ are allowed as transition functions. Clearly, such an atlas $\A_+$ may be completed with orientation-reversing projective charts to get a full projective atlas $\A$. Conversely one can extract from a projective atlas $\A$ on $C$ the orientation-preserving charts $\A_+$ (and also the corresponding atlas $\A_-$ for $C^o$, the same curve endowed with the other orientation). The only question is whether the same curve $C$, the projective structures defined by $(C,\A_+)$ and by $(C^o,\A_-)$ are equivalent {\em as oriented projective structures}, i.e. whether there exists an orientation-reversing automorphism $\tau$ of $(C,\A_+)$. We will answer this question separately for open, resp. closed curves at the end of this Section, resp in Proposition \ref{orrev}.

The developing map $D$ is identified (up to a diffeomorphism of $C$) with its image $U\subset\tp$. For $\rho$, there are two possible cases: $C$ is an open curve (in which case $\rho$ is trivial), and $C$ is a closed curve, when $\rho$ is identified with its image in $\tsl$, an infinite cyclic group acting freely and properly on $U$.

In the latter case, one can further choose from the two possible generators of $\rho(\pi_1(C))$:\begin{ede}\label{specg} For a closed projective curve, with developing map $D:\tl C\ra\tp$ and holonomy $\rho:\pi_1(C)\ra\tsl$, the generator of $\rho(\pi_1(C))$ such that
\be\label{pos-gen} \tl A(x)>x,\ \forall x\in U=D(\tl C),\ee
is called {\em the positive generator} of the holonomy.\end{ede}
We retrieve the following result (used implicitly in \cite{gold} and \cite{ku}):
\begin{ath}\label{clas-g} The oriented projective structures on $C\simeq S^1$, up to diffeomorphism, are in 1--1 correspondence with the equivalence classes of pairs $(U,\tl A)$ where $U\subset \tp$ is an interval and $\tl A\in\tsl$ such that $\tl A(U)=U$ and \eqref{pos-gen} holds, modulo the equivalence relation
  \be\label{equivUA}(U,\tl A)\sim (U',\tl A')\Longleftrightarrow \exists \tl B\in \tsl:\ U'=\tl B(U)\mbox{ and }\tl A'=\tl B\tl A\tl B^{-1}.\ee\end{ath}
To complete the classification we will need to study in detail the action of $\tsl$ on $\tp$,  starting from the action of $\SL(2,\R)$ on $\rp^1$.

\subsection{The action of $\tsl$ on $\tp$}

Denote
by $\pj:\R \to \rp^1,$  $\pj(t):=[\cos t:\sin t]$. Clearly, this is a realization of the universal covering of
$\rp^1$. In order to underline the projective structure of the domain of definition of $\pj$, we will always denote it by $\tp$ instead of $\R$. Note that the affine chart $\{[x:1]\ ,\ x\in\R\}\subset\rp^1$, which is the complement of the "point at infinity" $[1:0]\in\rp^1$, is isomorphic via $\pj$ with each of the intervals $(k\pi,(k+1)\pi)\subset\tp$, $k\in\Z$.

The action of
$\PSL(2,\R)$ on $\rp^1$ defined by the homographies lifts to a faithful,
transitive action of $\tsl$ on
$\tp$. The center of $\tsl$ is isomorphic to $\Z$ and coincides with the kernel of the projection $\Phi:\tsl\ra\PSL(2,\R)$, and also with
the automorphism group of the Galois covering $\pj:\tp\ra\rp^1$. We clearly have
\be\label{pj phi}\pj(\tl A(x))=\Phi(\tl A)(\pj(x)),\ \forall x\in\tp,\ \forall\tl A\in\tsl.\ee

We will also consider the "intermediate" covering $\bar\pj:\tp\ra S^1\subset\R^2$, $\bar\pj(t):=(\cos t,\sin t)$, because through it we can also see the intermediate projection $\bar\Phi:\tsl\ra\SL(2,\R)$, which is the double cover of $\PSL(2,\R)$: indeed, an element $\tilde A\in\tsl$ (which therefore acts on $\tp$) is mapped to a matrix $A:=\bar\Phi(\tl A)\in\SL(2,\R)$ such that $\bar\pj\circ\tl A=A\circ\bar\pj$, where $A\in\SL(2,\R)$ acts on $S^1$. Now, the action of $\SL(2,\R)$ on $S^1$ is given by
$$\SL(2,\R)\times S^1\ni(A,x)\mapsto \frac{Ax}{\|Ax\|}\in S^1,$$
where $\|x\|$ is the Euclidean norm of a vector in $\R^2$.
The expression is nonlinear unless $A\in \SO(2)$, the maximal compact subgroup of $\SL(2,\R)$:
\begin{prop}\label{elipt} The translations $\tl\ell_\a:\tp\ra\tp$ with $\a\in\R$, defined by $\tl\ell_\a(x):=x+\a$, are exactly the elements in $\tsl$ that project over the rotation matrices $\ell_\a:=\begin{pmatrix}\cos\alpha&-\sin\alpha\\
                            \sin\alpha&\cos\alpha\end{pmatrix}$ 
acting on $S^1$, in the sense that $\bar\pj\circ\tl\ell_\a=\ell_\a\circ \bar\pj$.
The center of $\tsl$ is equal to $\{\tl\ell_{k\pi}\ ,\ k\in\Z\}=\ker(\Phi)$. \end{prop}
It is well-known that $\SL(2,\R)$ admits the partition $\Hyp\cup\Para\cup\Ell^0\cup\mathcal{Z}$ defined by
$$\Hyp:=\{A\in\SL(2,\R)\ ,\ ,\tr A|>2\},\ \Ell^0:=\{A\in\SL(2,\R)\ ,\ ,\tr A|<2\}, $$
which are both open subsets formed of {\em hyperbolic}, resp. {\em properly elliptic} elements,  $\mathcal{Z}=\{\pm I_2\}$ is the center of $\SL(2,\R)$ (called {\em central}, or {\em degenerate elliptic} elements of $\SL(2,\R)$, cf. \cite{laz}) and $\Para:=\{A\in\SL(2,\R)\sm\mathcal{Z}\ ,\ ,\tr A|=2\}$ is the set of {\em parabolic} elements.

Let us give a classical result about the set of representatives of the conjugacy classes in $\SL(2,\R)$:
\begin{prop}\label{conjsl}
Every matrix $A\in\SL(2,\R)$ is conjugated to one of the following:\bi
\item $\pm H(\l):=\begin{pmatrix}\l&0\\ 0&\l^{-1}\end{pmatrix},\ \l\in (0,1)$;
\item $\pm P^\pm:=\begin{pmatrix}1&\pm1 \\ 0&1\end{pmatrix}$;
\item $\pm \ell_\a$, $\a\in (0,\pi)$;
\item $\pm\id$.\ei
Moreover, the above matrices represent different conjugacy classes. 
\end{prop}
\begin{proof} In the hyperbolic, resp. properly elliptic cases, the matrix $A$ has two distinct eigenvalues, therefore it is either real-diagonalizable, hence conjugated to a matrix of the form $\hyp(\l)$, with $\l\not\in \{-1,0,1\}$, or it is complex-diagonalizable, hence conjugated in $\SL(2,\R)$ to a matrix of the form $E_\a$, $\a\in (-\pi,0)\cup (0,\pi)$.

In the hyperbolic case, we have $JH(\l)J^{-1}=H(\l^{-1})$, where $J:=E_{\pi/2}$, which is not surprising as it has the same eigenvalues as $H(\l)$.

For $|\tr A|=2$, we have either a diagonalizable matrix, hence $A=\pm\id$, or $A$ is conjugated to $\pm P(x)$, which is defined as
\be\label{hypara}P(x):=\begin{pmatrix}1&x\\ 0&1\end{pmatrix},\ x\in\R.\ee
On the other hand $H(\l)P(x)H(\l^{-1})=P(\l^{-2}x)$, so by choosing $\l=\sqrt{|x|}$ we obtain that $A$ is conjugated to $\pm P^\pm$.

We need to show that two different representatives in the above list are not conjugated to each other. First, as the trace doesn't change by conjugation, we only need to prove that matrices of the same type (hyperbolic, parabolic, properly elliptic and central) are not conjugated to one another.

For hyperbolic matrices this is clear as $\tr A$ determines uniquely $\l$ such that $|\l|\in (0,1)$ and $\l+\l^{-1}=\tr A$.

For elliptic matrices we only need to show that $E_\a$ is not conjugated to $E_{-\a}$, for $\a\in (0,\pi)$. Suppose $B\in \SL(2,\R) $ such that $BE_\a B^{-1}=E_{-a}$. For a non-zero vector
$X\in\R^2$, we can write 
$$X\wedge E_{-\a}X=BB^{-1}X\wedge B E_\a B^{-1}X = \det B(B^{-1}X\wedge E_\a B^{-1}X)=Y\wedge E_\alpha Y,$$ 
for $Y:=B^{-1}X$. This is a contradiction since for any non-zero vector
$X\in\R^2$, the bi-vector $X\wedge E_{\a}X$ is a positive multiple of the volume element $e_1\wedge e_2$, whereas $X\wedge E_{-\a}X$ is a negative multiple of $e_1\wedge e_2$. 

For parabolic matrices we need to show that $P^+$ and $P^-$ are not conjugated in $\SL(2,\R)$. If this were possible, then there would be a basis $X_1,X_2$ of $\R^2$ such that $P^-X_1=X_1$ (this already implies $X_1=\begin{pmatrix}a\\ 0\end{pmatrix}$, with $a\in\R^*$), and $P^-X_2=X_1+X_2$. Denoting $X_2=\begin{pmatrix} b\\ c\end{pmatrix}$, with $c\ne 0$, this last equation reads
$$\begin{pmatrix}b-c\\ c\end{pmatrix}=\begin{pmatrix}a+b\\ c\end{pmatrix},$$
thus $a=-c$ which implies that the basis $X_1,X_2$ is not positive, i.e., the basis change is made by a matrix with negative determinant, contradiction.
\end{proof}
The sign ambiguities are removed when considering the homographies associated to matrices in $\SL(2,\R)$, i.e. the conjugacy classes in $\PSL(2,\R)$:
\begin{prop}\label{conjpsl} The group $\PSL(2,\R)$ admits the following partition; $\PSL(2,\R)=\hat\Hyp\cup\hat\Para\cup\hat\Ell^0\cup\{\id\}$ invariant by conjugation, and a full set of representatives is given by, respectively:\bi\item $\hat H(\l)$, $\l\in (0,1)$;
\item $\hat P^\pm $;
\item $\hat E_\a$, $\a\in (0,\pi)$;
\item $\id$.\ei\end{prop}
\begin{remark}\label{parfix}A fixed point for the action of $A\in\SL(2,\R)$ on $\rp^1$ corresponds to a real eigenspace of $A$ as a matrix, hence the partitions above can also be retrieved by looking at the number of fixed points in $\rp^1$ of the respective elements: two for hyperbolic, one for parabolic, and none for properly elliptic elements. Central elements act trivially on $\rp^1$.  
\end{remark}
Pulling back through $\phi$ the above partition, we get
\be\label{part}\tsl=\tl\Hyp\cup\tl\Para\cup\tl\Ell^0\cup\tilde{\mathcal{Z}},\ee
a partition in hyperbolic, parabolic, properly elliptic and central elements (which, together, form the elliptic elements) of $\tsl$. 

In order to give in $\tsl$ a full set of representatives of the conjugacy classes, we need the following
\begin{lem}\label{conjcover} Let $\pj:\tl G\ra G$ be the universal cover of a Lie group $G$, and denote by $Z:\ker\pj$, consisting of central elements in $\tl G$. There exists a 1--1 correspondence between conjugacy classes in $\tl G$ and pairs $([g],z)$, with $g\in G$, $z\in Z$, where $[g]$ is the conjugacy class of $g$ in $G$.\end{lem}
\begin{proof} Let $\tl g_2=\tl h \tl g_1\tl h^{-1}$, for $\tl g_{1,2},\tl h\in\tl G$ and set $g_{1,2}:=\pj(\tl g_{1,2})$, $h:=\pj(\tl h)$. Clearly $g_1$ is conjugated to $g_2$ in $G$. Conversely, if we have $hg_1h^{-1}=g_2$, the element $\tl h\tl g_1\tl h^{-1}$ is the same for every lift $\tl h$ of $h$ via $\pj$, because two such lifts differ by a central element. Therefore there exists $z\in Z$ such that $\tl g_2=z\tl h \tl g_1\tl h^{-1}$.

This means that a conjugacy class in $\tl G$ is determined by a conjugacy class in $G$ and an element in $Z$. Note that in order to associate to $[\tl g]$ a pair $([g],z)$ we need to fix a (non-canonic) lift of each representative $g\in [g]\subset G$ to an element $\tl g\in\pj^{-1}(g)$. 
\end{proof}
We consider now lifts of the representatives from Proposition \ref{conjpsl}:

The translations in $\R=\tp$ are properly elliptic if and only if $\a\not\in\pi\Z$, they are central otherwise, cf. Proposition \ref{elipt}. From the Lemma above we conclude that $\{\tl\ell_\a\ ,\ \a\in\R\}$ is a full set of representatives of conjugacy classes of elliptic elements $\tl\Ell^0\cup\tl{\mathcal{Z}}$.

For $\tl A\in\tl\Hyp\cup\tl\Para$, $\Phi(\tl A)=\hat A\in\hat\Hyp\cup\hat\Para$ has two, resp. a fixed point on $\rp^1$ and the following can be said about $\tl A$:
\begin{lem}\label{tl a0} Let $\tl A\in\tsl$ and $\hat A:=\Phi(\tl A)\in\PSL(2,\R)$ be its projection. If $\tl A$ fixes $x\in\tp$, then it fixes all points in $\pj^{-1}(\pj(x))=\{x+k\pi\ ,\ k\in\Z\}$. If $\hat A$ has, in addition to $\pj(x)$, also another fixed point $z\in\rp^1$, then $\tl A(y)=y$ for all $y\in\pj^{-1}(z)$. 

Conversely, if $\hat A\in\SL(2,\R)$ has a fixed point in $\rp^1$, then there exists a unique element $\tl A_0\in \tsl$ such that $\Phi(\tl A_0)=A$ and $\tl A_0$ fixes one (hence all) points in $\pj^{-1}(z)$, for any fixed point $z$ of $\hat A$. Moreover, $\Phi^{-1}(\hat A)=\{\tl A_k:=\tl A_0\tl\ell_{k\pi}\ ,\ k\in\Z\}$. \end{lem}
\begin{proof} 
Let $x+k\pi=\tl\ell_{k\pi}(x)$ be any element of $\pj^{-1}(\pj(x))$. Then  $\tl A(x+k\pi)=\tl A(\tl\ell_{k\pi}(x))=\tl\ell_{k\pi}(\tl A(x))=\tl\ell_{k\pi}x=x+k\pi$ for any $k\in\Z$, because $\tl\ell_{k\pi}$ are in the center of $\tsl$. This proves the first claim. 

For the second claim, note that $\pj^{-1}(z)$ is a $\tl A$-invariant subset of $\tp$, and it consists of one point $y_k$ for each interval $(x+k\pi,x+(k+1)\pi)\subset\tp$. The endpoints of this interval are fixed points of $\tl A$, which is also a monotone increasing map from $\tp=\R$ to itself. That means that $\tl A(y_k)$ also belongs to the same interval, hence  $\tl A(y_k)=y_k$. 

Let us now consider $\hat A\in\PSL(2,\R)$ that has a fixed point $z\in\rp^1$. Then any $\tl A\in\Phi^{-1}(\hat A)$ maps $\pj^{-1}(z)$ to itself, in particular, for every $x\in \pj^{-1}(z)$ there exists an integer $k$ (which is actually independent of $x$) such that $\tl A(x)=x+k\pi=\tl\ell_{k\pi}(x)$. Therefore $\tl A=\tl A_k$ as in the statement.\end{proof}
\begin{remark}\label{a0} For an element $A\in\SL(2,\R)$ acting (as $\hat A\in\PSL(2,\R)$) with fixed points on $\rp^1$, we call $\tl A_0$ as in Lemma \ref{tl a0} the {\em canonical lift} of $A$, resp. $\hat A$ in $\tsl$. Note that the elements $\tl A_k$, $k\in\Z^*$, have no fixed points on $\tp$. \end{remark}

\begin{remark}\label{incr} The action of the connected group $\tsl$ on $\tp$ is by monotone increasing diffeomorphisms. The orientation-reversing diffeomorphism of $\tp\equiv\R$, defined by $\tl\tau(t):=-t$ also projects via $\pj$ to a diffeomorphism of $\rp^1$, namely the map $[x:y]\mapsto [x:-y]$ which actually is the homography $f_T$, defined by $T:=\begin{pmatrix}1&0\\ 0&-1\end{pmatrix}\in\GL(2,\R)$, hence $\tau$ is an orientation-reversing projective automorphism of $\tp$. Note that, if $\tl A_0(0)=0$, then $\tl\tau\circ\tl A_0\circ\tl \tau$ is again an orientation-preserving diffeomorphism on $\tp$ that fixes $0$, thus it is equal to some $\tl B_0$ and $Fix(\tl B_0)=-Fix(\tl A_0)\subset\tp=\R$. Consequently, conjugation with $\tl\tau$ maps $\tl\Hyp$, $\tl\Para$, $\tl\Ell^0$ and $\tl{\mathcal{Z}}$ to themselves. Here $Fix(F)$ denotes the set of fixed points of a diffeomorphism $F:\tp\ra\tp$.\end{remark}

\subsection{Open projective curves } An open, oriented, projective curve $C$ can
be seen, via a developing map, as an interval in $\R=\tp$, and we need to classify them up to an isomorphism in $\tsl$. There are three types of open intervals:
\bi\item $\tp$ itself;
\item a half-line;
\item a bounded interval $(\a,\b)$.\ei

Obviously, any element $\tl A\in\tsl$ acts by bijections of $\tp$,
thus $\tp$ can not be isomorphic, as a projective curve, to a
half-line or some bounded interval. On the other hand, because $\tl A:\tp\ra\tp$
is monotone increasing (see Remark \ref{incr}), we have
$$\tl A\left( (-\infty,p_0)\right)=(-\infty,\tl A(p_0)),$$
hence, on one hand, such a {\em left} half-line is not isomorphic to any {\em right}
half-line or to any bounded interval and, on the other hand, any two left (or right)
half-lines are isomorphic (because $\tsl$ acts transitively on
$\tp$).
Thus, $\tp$, left and right half-lines are three distinct isomorphism classes of
oriented open projective curves.

Let us consider the remaining type (bounded
intervals).

\begin{ede}\label{windo} Let $U:=(a,b)\subset\tp$, with
  $a,b\in\R$, $a<b$.
\bi\item If $b=a+k\pi$, $k\in\N$, or, equivalently $b=\tilde
\ell_{k\pi}(a)$, we say that the interval $U$ has {\em
  winding number} $k\in\N$, and write $W(U)=k$.
\item If $a+k\pi<b<a+(k+1)\pi$, (or equivalently $\tl\ell_{k\pi}(a)<b< \tl\ell_{(k+1)\pi}(a)$)  with $k\in\N$, we say that
  the interval $U$ has {\em winding number} $k+\tfrac12$, $k\in\N$ and write $W(U)=n+\tfrac12$.\ei
\end{ede}
For example, an interval projecting via $\pj$ on a maximal affine chart, or $\rp^1$ without a point, e.g. the 
interval $(0,\pi)$, has winding number $1$, and an interval projecting on an affine half-line inside an affine chart, e.g. the interval $(0,\tfrac\pi2)\subset\tp$, has, by the above convention, winding number $\tfrac12$. 

In general the winding number of an interval $U\subset\tp$ is related to the number in $U\cap\pj^{-1}(x)$ (how many times $U$ is {\em wrapped} around $\rp^1$), as follows;  $\{\sharp (I\cap\pj^{-1}(x))\ ,\ x\in\rp^1\}$ always consists of two non-negative integers $k,k+1$, the distinction is made by how many points $x\in\rp^1$ attain the lower number above: if there is only one point, then $W(U)=k+1$, if there are at least two (hence infinitely many), then $W(U)=k+\tfrac12$.

 We can now prove the following:
\begin{prop}\label{open}
The winding number (with values in $\tfrac12\N$) is the unique
invariant of a bounded interval of\ \, $\tp$ under the action of \, $\tsl$:
For $U$ and $U'$ bounded intervals of $\tp$, there exists an
element $\tilde A\in\tsl$ such that $\tilde A(U)= U'$ and if and only
if these intervals have the same winding numbers.  
\end{prop}
\begin{proof} By mapping $U,U'$ by appropriate elements of $\tsl$, we can assume $U=(0,p)$ and $U'=(0,p')$. Any further action on $U$ needs to be by $\tl A$ that fixes $0\in\tp$. 

Assume that there exists $\tl A\in\tsl$ such that $\tl A(0)=0$ and $\tl A(p)=p'$. If $W(0,p)=k\in\N$
then $p=k\pi$ and $p'=\tl A(k\pi)=k\pi=p$ as implied by Lemma \ref{tl a0}. 

If the winding number of $(0,p)$ is $k+\tfrac12$ then $p\in (k\pi,(k+1)\pi)$ and $p'=\tl A(p)$ lies in the same interval (see the proof of Lemma \ref{tl a0}), hence $W(0,p')=k+\tfrac12$ as claimed.

Conversely, assume that $(0,p)$ and $(0,p')$ have the same
winding numbers. If this number is $k\in\N$ then $p=p'=k\pi$ and we may even take $\tl A$ to be the identity $\tl \ell_0\in\tsl$.

If $W(0,p)=W(0,p')=k+\tfrac12$, then $p,p'\in (k\pi,(k+1)\pi)$. Their projections $\pj(p),\ \pj(p')\in \rp^1$ are equal to $[x:1]$, resp $[x':1]$, simply because they are both different from $\pj(k\pi)=\pj(0)=[1:0]$. We have $[x':1]=\hat\para(x'-x)([x:1])$, where $\para(x'-x)$ is the parabolic matrix defined in \eqref{hypara}. Lemma \ref{tl a0} implies that $\tl\para_0(x'-x)(0)=0$ and $\tl\para_0(x'-x)(p)=p'$ for $\tl\para_0(x'-x)$ the canonical lift of $\para(x'-x)$, see Remark \ref{a0}.

Therefore $\tl\para_0(x'-x)$  maps $(0,p)$ onto $(0,p')$ as claimed.
\end{proof}

We see that the isomorphism class of an open, oriented projective curve is determined by its winding number or, if the latter is infinite, there are three possible isomorphism classes: $(0,\infty)$, $(-\infty,0)$, $\R\subset\rp^1$.

For non-oriented open projective curves, note that the map $s\mapsto -s$ on $\tp$ (which induces an orientation-reversing projective automorphism of $\tp$) does
not change the winding number, but establishes an isomorphism between
a left- and a right-half-line. Therefore, we have (see  also \cite{ku}, \cite{gold}):
\begin{cor}\label{openn} The moduli space $\mathcal{M}^o$ of isomorphism classes of
  connected, open projective curves (disregarding the orientation) is countable: 
$$\mathcal{M}^o=\{\infty,H\}\cup\tfrac{1}{2}\N^*.$$
The point $\infty$ corresponds to $\tp$, $H$ represents the
isomorphism
class of any half-line in $\tp$, and any
number $x$ in $\frac{1}{2}\N$ represents the isomorphism class of a bounded interval in
$\tp$ of winding number $x$. 
\end{cor}

\section{The isomorphism classes of closed projective curves}
In order to classify, up to isomorphism, the projective structures on a closed curve, we need to refine Theorem \ref{clas-g} using our knowledge of the action of $\tsl$ on $\tp$.

\subsection{The conjugacy classes in $\tsl$}
Using Lemma \ref{conjcover}, we refine the partition \eqref{part} by decomposing it further in conjugacy classes, and give for each of them a standard representative:
\begin{ath}\label{conjtsl} The conjugacy classes of elements in $\tsl$ are the following, with respective representatives:
\bi\item $\displaystyle{\tl\Hyp =\bigcup_{k\in\Z,\ \l\in (0,1)}\tl\Hyp_k(\l)}$, represented by $\tl H_k(\l):=\tl H_0(\l)\tl E_{k\pi}\in \tl\Hyp_k(\l)$, where $\tl H_0(\l)$ is the canonical lift of $\hat H(\l)\in\PSL(2,\R)$;
\item $\displaystyle{\tl\Para=\bigcup_{k\in\Z}\left(\Para_k^+\cup\Para_k^-\right)}$, represented by $\tl P_k^\pm:=\tl P_0^\pm\tl\ell_{k\pi}\in\tl\Para_k^\pm$, where $\tl P_0^\pm$ is the canonical lift of $\hat P^\pm\in\PSL(2,\R)$;
\item $\displaystyle{\tl\Ell^0=\bigcup_{\a\in\R\sm\pi\Z}\tl\Ell_\a}$, represented by $\tl\ell_\a\in\tl\Ell_\a$;
\item $\tl{\mathcal{Z}}=\{\tl\ell_{k\pi}\ ,\ k\in\Z\}$, formed by conjugacy classes containing only one element each.
\ei\end{ath}

\begin{remark} The smooth map $\SL(2,\R)\ni A\mapsto \tr A\in\R$ is a submersion when restricted to $\SL(2,\R)\sm\mathcal{Z}$, and the connected components of the level sets (which are, therefore, closed surfaces in $\SL(2,\R)\sm\mathcal{Z}$) are exactly the conjugacy classes. We can define the map $\k:\tsl\ra\R$, $\k(\tl A);=\tr(\pj(\tl A))$, which is also a submersion on $\tsl\sm\tl{\mathcal{Z}}$, and its level sets are (grouped pairwise by the notation $\k^{-1}(\pm a):=\k^{-1}(a)\cup \k^{-1}(-a)$, for $a>2$, $0\le a<2$ and resp. $a=2$):
$$\k^{-1}\left(\pm\tfrac{\l^2+1}{\l}\right)=\bigcup_{k\in\Z}\tl\Hyp_k(\l),\quad \k^{-1}(\pm 2\cos\a)=\bigcup_{k\in\Z}\tl\Ell_{\a+k\pi},\quad \k^{-1}(\pm 2)=\tl\Para.$$
Here, in order to write each component of a level set exactly once, we make the convention that $\a\in(0,\pi)$ and $\l\in(0,1)$. 
\end{remark}
\subsection{Classification of projective structures on closed curves (with or without a fixed orientation)}
We will simplify the data associated with a projective structure, replacing a pair $(U,\tl A)$ as in Theorem \ref{clas-g} with $\tl A$ alone. As $\tl A$ from the above mentioned theorem satisfies the inequality \eqref{pos-gen}, we define a subset of $\tsl$ where $\tl A$ is contained:
\begin{ede}\label{tsl+} The subset $\tsl_+$ is the set of elements $\tl A$ such that there exists at least one point $x\in\tp$ such that $\tl A(x)>x$.\end{ede}
Clearly, all elements conjugated to $\tl A\in\tsl_+$ are also in $\tsl_+$. For every $\tl A\in \tsl$, $\tl A$ or its inverse are in $\tsl_+$ (sometimes both, see below), and if $\tl A$ has no fixed points on $\tp$, then exactly one of $\tl A$ and $\tl A^{-1}$ belongs to $\tsl_+$. The element $\tl \hyp_0(\l)$ admits points on $\tp$ that satisfy $\tl \hyp_0(\l)(x)>x$ and also points satisfying $\tl \hyp_0(\l)(y)<y$.  On the other hand, we have $\tl\para_0^+(x)\le x$ and $\tl\para^-_0(x)\ge x$ for all $x\in\tp$. We conclude:
\be\label{slplus}\tsl_+=\bigcup_{k\in\N,\ \l\in (0,1)}\tl\Hyp_k(\l)\cup\bigcup_{\a>0}\tl\Ell_\a\cup \bigcup_{k\in\N}\tl\Para_k^-\cup \bigcup_{k\in\N^*}\tl\Para_k^+.\ee
We show now that an isomorphism class of an oriented closed projective curve $C$ is encoded in the positive generator $\tl A$ of its holonomy group, cf Definition \ref{specg}. Note that if $\tl A$ has no fixed points, then there is nothing to show since in that case $C=\tp/{\langle\tl A \rangle}$.

\begin{prop}\label{dev-sl+} Let $\tl A\in\tsl_+$ and let $a<b$ be fixed points of $\tl A$ such that $\tl A(x)>x$, $\forall x\in U:=(a,b)\subset\tp$. Then, every other interval $U'$ with same property is of the form $U'=(a+k\pi,b+k\pi)$ for some $k\in\Z$. Moreover, the projective curves $U/{\langle\tl A \rangle}$ and $U'/{\langle\tl A \rangle}$ are isomorphic.\end{prop}
\begin{proof} 
The statement is invariant by conjugation in $\tsl$ since the translation by $k\pi$ in $\R$ is given by the central element $\tl E_{k\pi}$, and conjugation in $\tsl$ preserves the inequalities in $\R=\tp$ (in the sense that if $\tl A=\tl B\tl A'\tl B^{-1}$ and $A(x)>x$ on some interval $U$, then $A'(x)>x$ on $\tl B(U)$.

If $\tl A\in\tsl_+$ has fixed points, we can thus assume up to conjugation that it is either equal to $\tl\para_0^-$, whose fixed points are $\pi\Z$, or to some $\tl\hyp_0(\l)$, $\l\in (0,1)$, whose fixed points are $\frac\pi2\Z$.

In the first case, the interval $U$ is necessarily of the form $(k\pi,(k+1)\pi)$, which trivially implies our claim.

In the second case, since $H(\lambda)$ maps an element $[\cos\alpha:\sin\alpha]\in \rp$ onto $[\lambda\cos\alpha:\lambda^{-1}\sin\alpha]$, whose argument is larger than $\alpha$ for $\alpha\in (0,\frac\pi2)$ and smaller than $\alpha$ on $(\frac\pi2,\pi)$, we obtain that $U$ is necessarily of the form $(k\pi,k\pi+\frac\pi2)$, and again the claim follows.

The last statement follows directly from the fact that the projective structures associated to the pairs $(U,\tl A)$ and $(\tl E_{k\pi}U,\tl A)$ are isomorphic by taking $\tl B=\tl E_{k\pi}$ in Theorem \ref{clas-g}.
\end{proof}
\begin{cor}\label{moduli} The moduli space of oriented projective structures on closed curves is the set of conjugacy classes of elements of $\tsl_+$.\end{cor}
We show now that we can drop the word "oriented" from the above result:
\begin{prop}\label{orrev} All closed projective curves admit orientation-reversing projective automorphisms. Therefore, the set of oriented projective structures on a closed curve $C$ coincides with the set of projective structures on $C$.\end{prop}
\begin{proof} It is enough to prove the statement for  particular representatives $\tl A$ of each conjugacy class in $\tsl_+$. If $\tl A$ has no fixed points on $\tp$, it suffices to construct an orientation-reversing map on $\tp$ (that induces on $\rp^1$ an orientation-reversing homography) such that the holonomy group $\lg\tl A\rg$ is preserved. In all cases we set $\tl\tau:\tp\ra\tp$ by $\tl\tau(x):=-x$ (which induces the matrix $T\in\GL(2,\R)$, see Remark \ref{incr},  and we claim that
\be\label{or-rev}\tau\tl\ell_\a\tau=\tl \ell_{-\a},\quad  (\tau\tl J)\tl\hyp(\l)(\tl J^{-1}\tl\tau)=\tau\tl\hyp_n(\l^{-1})\tau =\tl\hyp_n(\l^{-1}),\quad \tau\tl\para_n^\pm\tau=\para_{-n}^\mp,\ \forall n\in\Z,\ee
where juxtaposition means composition of projective maps on $\tp$.
If $\tl A$ is one of these elements, its conjugate by $\tau$, resp. $(\tau\tl J)$, would be exactly its inverse, thus the group $\lg\tl A\rg$ would be  preserved. 

To prove \eqref{or-rev} we proceed case by case: for $\tl A=\tl\ell_\a$ this is an identity involving affine transformations on $\R=\tp$: $$\tau\tl\ell_\a\tau(x)=-(-x+\a)=x-\a=\tl\ell_{-\a}(x),\ \forall x\in\tp.$$
For the hyperbolic and parabolic cases we first prove the identities for $k=0$, i.e., $\tl A=\tl A_0$ is the canonical lift of a matrix $A\in\SL(2,\R)$ of hyperbolic, resp. parabolic type. It is straightforward that the right hand sides of \eqref{or-rev} also have fixed points on $\tp$, hence the identities are implied by the corresponding identities of matrices, i.e.
$$TJ^{-1}H(\l)JT=TH(\l^{-1})T=H(\l^{-1}),$$
for which the first identity has already appeared in the proof of Proposition \ref{conjsl} and the second follows because $T$ and $H(\l^{-1})$ are both diagonal, hence they commute, and
$$TP(\pm1)T=P(\mp1),$$
which is again a very simple computation.

In general $\tl A=\tl A_0\tl\ell_{k\pi}$, $k\in\Z$, so we apply the already proved identities for $\tl\ell_{k\pi}$ and for $\tl A_0$, thus finishing the proof of \eqref{or-rev}. 
    \end{proof}
\subsection{Classification of Laplace structures on closed curves}
The classification of Laplace structures on closed curves, that can be also obtained by following Theorems \ref{main}, \ref{clas-g} and Proposition \ref{dev-sl+}, can be stated by exhibiting a natural map that behaves well under deformations:
\begin{ath}\label{clasif} There exists a natural map $\Xi$ from the set of equivalence classes of real-valued, $1$-periodic functions $F:\R\ra\R$, up to the action of the diffeomorphism group of $\R/\Z$ described in \eqref{0ord},  to the conjugacy classes of elements in $\tsl_+$. This map is defined as $F\mapsto \tl R(1)$, where $R$ is the solution of the first order differential system $R'=R\begin{pmatrix}0& -F\\ 1&0\end{pmatrix}$ with the initial condition $R(0)=I_2$, and $\tl R:[0,1]\ra\tsl$ is the lift of the path $R:[0,1]\ra\SL(2,\R)$ with $\tl\R(0)=\tl\ell_0$.

The map $\Xi$ is bijective and continuous for the standard topologies of the corresponding quotient spaces.\end{ath} 
\begin{proof}
  The system 
  \be\label{R'} \begin{cases}
      R'=RQ\\ R(0)=I_2
  \end{cases}\qquad\mbox{ where } Q:=\begin{pmatrix}0& -F\\ 1&0\end{pmatrix},\ee
  is the matrix version of the same first order linear ODE associated to Hill's equation $x''+Fx=0$, i.e. for a solution $x(t)$ of Hill's equation, the (row) vector-valued map $X(t):=(x(t)\ x'(t))\in\R^2$ (seen as row matrices) satisfies
  $X'=XQ$, and conversely, the second entry of a solution to $X'=XQ$ is the derivative of the first entry $x$, which satisfies the Hill's equation. So the first column of $R$ consists of two solutions $x_1, x_2$ of $x''+Fx=0$, and the second column of their derivatives, thus the determinant of $R$ is equal to the {\em Wronskian} $Wr(x_2,x_1):=x_2'x_1-x'_1x_2$ which is constant, thus equal to 1 (the value of $\det R(0)$). Thus $R:\R\ra \SL(2,\R)$ and it admits a well-defined lift to $\tsl$, denoted with $\tl R$. From the standard theory of ODE's, the map 
$$F\mapsto \tl R(1)\in\tsl$$
is well-defined and continuous. Clearly, the first column \be\label{psi} \begin{pmatrix}x_1\\ x_2\end{pmatrix}=R\begin{pmatrix}1\\ 0\end{pmatrix}\ee
of $R$ determines a map $\psi:\R=\tl{S^1}\ra\rp^1$ such that $\psi(0)=[1:0]$, that lifts to a map $\tl\psi:\R\ra\tp$, such that $\tl\psi(0)=0$.  The fact that the Wronskian $Wr(x_2,x_1)$ is positive implies that the map $\tl\psi$ has positive derivative.  In particular, $\tl\psi(1)>\tl\psi(0)=0$. The map $\tl\psi:\R\ra\tp$ is thus not only an immersion, but a monotone increasing one, and this is a developing map for the associated projective structure, cf. Theoerm \ref{main} and Remark \ref{iso-lap-proj}. 

We claim that the holonomy corresponding to the developing map $\tl\psi$ is generated by $\tl R(1)$: indeed, we need to prove that $\tl\psi(t+1)=\tl R(1)(\tl\psi(t))$ or, equivalently, that $\psi(t+1)=\hat R(1)\psi(t)$ for all $t\in\R$. This follows from the fact that $t\mapsto R(1)R(t)$ and $t\mapsto R(t+1)$ satisfy the same equation $R'=RQ$ with the same initial condition at $t=0$.

From \eqref{psi} we have that $\tl\psi(1)=\tl R(1) (0)$ (the element $\tl R(1)$ of $\tsl$ applied to $0\in\tp$), thus $\tl R(1)\in\tsl_+$. Consequently, we get a continuous map 
\be\label{Xi}\Xi:\{ F:\R\ra\R\ ,\ F \mbox{ smooth and }F(t+1)=F(t),\ \forall t\in\R\}\longrightarrow \tsl_+\ee
which is equivariant for the action of the diffeomorphism group of $\R/\Z$ on the left side, resp. conjugation in $\tsl$ on the right side, cf. Theorem \ref{main}.

Note that $\Xi$ is a composition of the correspondence between Laplace and projective structures from Theorem \ref{main} with the one from Theorem \ref{clas-g} (which associates to a projective structure an equivalence class of a pair $(U,\tl A)$)  and with the one from Proposition \ref{dev-sl+} (which drops the interval $U$ that turns out to be determined modulo $\pi\Z$ by $\tl A$). Because we know from that the correspondences above are  all 1--1, we conclude that $\Xi$ is bijective as well.
\end{proof}

\subsection{Geometric properties of closed projective curves}
We will describe some of the geometric properties of a closed projective curve that allow us to distinguish various isomorphism classes. The first property is whether the developing map is surjective or not. Recall that by Corollary \ref{moduli}, a projective curve is identified to a conjugacy class of the adjoint action of $\tsl$ on $\tsl_+$, and those corresponding to non-surjective developing maps are $\tl\Para_0^-\cup\tl\Hyp_0(\l)$, $\l\in(0,1)$.
\begin{prop}\label{afin}
The curves from $\tl\Para_0^-\cup\tl\Hyp_0(\l)$, $\l\in(0,1)$, are the projective curves that are
also {\em affine}, i.e. one can extract from the atlas of projective charts a smaller atlas of affine charts or, equivalently, it admits a global connection or which the corresponding Laplace operator is the Hessian. The distinction between $\tl\Para_0^-$ and $\tl\Hyp_0(\l)$ is that, in the first case, the above mentioned connection is unique and admits a global parallel section, and in the second case there are two such connections and none admits a parallel section. \end{prop}
\begin{proof} First note that if there exists $\nb$ on $C$ such that $\LL=\Hss^\nb$ on $C$, then every {\em local} $\nb$-parallel nontrivial section of $TC$ induces a {\em local} parametrization that is projective, so there exists a sub-atlas formed by those parametrizations whose speed vector field is $\nb$-parallel, see Theorem \ref{main}. The transition maps are affine, because their second derivative vanishes (which of course implies the vanishing of the Schwarzian). These parametrizations can be globally defined on $\tl C$ which is therefore sent through the developing map to an interval that projects diffeomorphically through $\pj$ onto its image in an affine chart of $\rp^1$. This implies that the projective structure is of type $\tl\Para_0^-$ or $\tl\Hyp_0(\l)$, $\l\in (0,1)$.

For $\tl A=\tl P_0^-$ we get that $(0,\pi)/\lg\tl P_0^-\rg$ is isomorphic with $\R/\Z$, as the element $\tl P_0^-$ induces on $\{[x:1]\ ,\ x\in\R\}\simeq\R$ the transformation $[t:1]\mapsto [t+1:1]$. For $\tl A=\tl H_0(\l)$, its action on $(0,\pi/2)$ induces the transformation 
$[x:1]\mapsto [\l^2 x:1]$, for $x>0$ (the image $\pj((0,\pi/2))=\{[x:1]\ ,\ x\in (0,\infty)\}$).

Note that for $\tl A=\tl P_0^-$, the image of $D$ is contained in a single affine chart in $\rp^1$, hence the uniqueness of the affine connection $\nb$, which is (after identifying $\tl C$ with $\R$, the trivial connection for which the derivative of the identity map (which is the affine parametrization) is the $\nb$-parallel (constant) vector field $1$. As noticed before, $\tl P_0^-$ acts on $\tl C\simeq\R$ by translations.

For $\tl A=\tl H_0(\l)$, the developing map has values in $(0,\pi/2)$ and its projection through $\pj$ on $\rp^1$ is contained in an infinity of affine charts, however only for two of them does $\tl A$ act by affine transformations: on those affine charts which contain at most (in fact, exactly) one of the two fixed points of $\hat A\in\PSL(2,\R)$ on $\rp^1$, namely on $\{[1:x]\ ,\ x\in\R\}\supset \pj(D(\tl C))=\{[1:x]\ ,\ x>0\}$ or on $\{[x:1]\ ,\ x\in\R\}\supset \pj(D(\tl C))=\{[x:1]\ ,\ x>0\}$. On each of these charts the holonomy is generated by $x\mapsto \l^2x$, resp $x\mapsto \l^{-2}x$, so the $\nb$-parallel (constant) functions/vector fields are not $\tl A$-invariant, hence not defined globally on $C$.
\end{proof}
The next invariant of an isomorphism class of closed projective curves is the set of winding numbers of the fundamental domains of $U/\lg\tl A\rg$:
\begin{prop}\label{windnumbers}
If $I:=[a,b)$ is a fundamental domain of $\tp/\langle\tl
A\rangle$ (thus $b=\tl A(a)$), then the winding number of $I$ depends
on $\tl A$ as follows:
\bi
\item If $\tl A$ is conjugated to $\tl E_\a$, $\a>0$, then $W(I)=n$ if
  and only if $\alpha=n\pi$, otherwise $W(I)=n+\frac12$, where $n\pi<\alpha<(n+1)\pi$;
\item If $\tl A$ is conjugated to $\tl P^+_n$ or $\tl
  P^-_n$, for $n\in\N^*$, then $n+\frac12\ge W(I)\ge n$, respectively
  $n-\frac12\le W(I)\le n$;
\item If $\tl A$ is conjugated to $\tl H_n(\lambda)$, $\l>1$, $n\in\N^*$,
  then $n-\frac12\le W(I)\le n+\frac12$.\ei
In cases {\rm (2)} and {\rm (3)}, the winding number of $(a,b)$ is
integer if and only if the projection $\pj(a)\in\rp^1$ is a fixed
point of the projection $\hat A\in \PSL(2,\R)$ of $\tl A$. Moreover, the
limits in the inequalities above are attained.
\end{prop}
\begin{proof} As the claims are invariant to conjugation, we will assume $\tl A\in\tsl_+$ is one of the representatives of the elliptic, parabolic or hyperbolic conjugacy classes.

Note that if $W(a,\tl A(a)):=W([a,\tilde A(a)))$ is an integer, then $\pj(a)=\pj(\tilde A(a))$ is a
  fixed point for $\hat A$ on $\rp^1$. 
  
In case (1), $\tl A$ is the translation with $\a$, hence $W(x,x+\a)=n$ if and only if $\a=n\pi$, $n\in\N^*$, and $W(x,x+\a)=n+\tfrac12$ if $\a\in (n\pi,(n+1)\pi)$, as claimed.

In case (2), $\tl A(x)=\tl P_n\pm=\tl P_0^\pm(x)+n\pi$, $n\in\N^*$ or $\tl A=\tl P^-_0$. Clearly, $\tl P_0^+(x)\le x$ for all $x\in\tp$, with equality for $x\in\pi\Z$, and $\tl P_0^-(x)\ge x$ for all $x\in\tp$, with equality for $x\in\pi\Z$. In both cases $\tl P_0^\pm(x)-x\in (-\pi,\pi)$ (because all intervals of the form $(k\pi,(k+1)\pi)$ are mapped to themselves), hence $W(x,\tl P_0^\pm(x)+n\pi)=n\mp \tfrac12$ if $x\in\tp\sm\pi\Z$ and $W(x,\tl P_0^\pm(x)+n\pi)=n$ if $x\in\pi\Z$, as claimed. 

We argue similarly in case (3): $\tl A(x)=\tl H_0(\l)(x)+n\pi$, and the difference $\eta(x):=\tl H_0(\l)(x)-x\in (-\pi,\pi)$, vanishes exactly for $x\in\tfrac\pi2\Z\subset \tp$. For $k\in\Z$ and  $x\in (k\pi,(k+\tfrac12)\pi)$, $\eta(x)>0$ and for $x\in ((k-\tfrac12)\pi,k\pi)$, $\eta(x)<0$. The corresponding winding numbers $W(x,\tl A(x))$ for these values of $x$ are therefore $n$, $n+\tfrac12$, resp. $n-\tfrac12$.
\end{proof}

\begin{ede}\label{windc}
 The {\em winding number} $W(C)$ of a closed curve $C\simeq
  \tp/\langle\tl A\rangle$ is $n\in\N^*$ if there exist $a\ne
  b:=\tl A(a)\in\tp$ such that $W([a,b))=n$. Otherwise
  $W(C):=W(I)\in\frac12\N^*\sm\N^*$ for any fundamental domain $I$ of
  $C$ as in the proposition above.\end{ede} 
\begin{remark}\label{explain} In \cite{ku}, and also in \cite{hit}, the authors did not distinguish between the classes $\tl\para_n^+$ and $\tl\para_n^-$ for the same $n\in\N^*$. Proposition \ref{windnumbers} produces a geometric invariant (see also the table below) that distinguishes these classes. \end{remark}

Another invariant of an isomorphism class of closed projective curves is its group of global projective
automorphisms:
\begin{prop}\label{autom} The orientation-preserving automorphism
  group $\Aut^+(C)$ of an oriented closed projective curve $C\simeq \tp/\langle
  \tilde A\rangle$ is $3$--dimensional (equal to $\tsl/\langle
  \tilde E_{n\pi}\rangle$) if and only if the curve is of class
  $\Ell_{n\pi}$, $n\in\N^*$. Otherwise, $\Aut(C)$ is a
    1--dimensional Lie group and:
\bi\item $\Aut^+(C)$ is connected and compact (i.e., isomorphic to $S^1$)
if and only if 
$$C\in\tl\Ell^0\cup\tl\Hyp_0\cup\tl\Para^-_0;$$ 
\item $\Aut^+(C)$ is non-connected and non-compact if and only if
$$C\in(\tl\Hyp\sm\tl\Hyp_0)\cup(\tl\Para\sm\tl\Para_0^-).$$\ei
In the first case, $C$ is a {\em homogeneous} projective
curve. In the second case, $C$ is non-homogeneous and $\Aut_0(C)$, the
connected component of the identity in $\Aut^+(C)$, has $n=W(C)$ fixed points 
if $C\in\Para^\pm_n$, respectively $2n=2W(C)$ fixed points if
$C\in\Hyp_n(\lambda)$. In these cases, the index of $\Aut_0(C)$ in $\Aut^+(C)$ is
$W(C)$.
\end{prop}
\begin{proof} In general, $\Aut^+(C)$ is equal to $N(\langle\tl A\rangle)/\langle\tl A\rangle$, where $N(\langle\tl A\rangle)$
denotes the normalizer of the discrete
  group $\langle\tl A\rangle$ in $\Aut (\tl C)$ (this latter group being equal to $\tsl$ if $\tl C\simeq \tp$, and an affine group otherwise). Moreover, $\mathrm{Aut}_0(C)$ 
is equal to the connected component of the identity in the centralizer of
$\tilde A$. Because the group $\lg\tl A\rg$ is infinite cyclic, an element $\tl B$ lies in its normalizer if an only if $B$ is in its centralizer or $\tl B\tl A\tl B^{-1}=\tl A^{-1}$. For $\tl A\in\tsl_+$ a standard representative of its conjugacy class, the latter is only possible if $\tl A=\tl H_0(\l)$.

If $\tl A\in\tl{\mathcal{Z}}\cap\tsl_+$ (i.e., $\tl A=\tl\ell_{n\pi}$, $n\in\N^*$), then the whole $\tsl$ is the centralizer of $\tl A$, and therefore $\Aut^+(\tp/\lg\tl A\rg)=\tsl/\lg\tl\ell_{n\pi}\rg$, and it is connected.

If $\tl A=\ell_\a$, $\a\in\R\sm\pi\Z$, then $\bar\pj(\tl A)=A=\ell_\a$, $\a\in (0,\pi)\cup (-\pi,0)$, has complex, non-real eigenvectors that are shared by any matrix $B\in\SL(2,\R)$ that commutes with $A$, i.e., $B=\ell_\b$, With $\b\in(-\pi,\pi]$. The centralizer of $\tl A$ is then $\{\tl\ell_\b\ ,\ \b\in\R\}$. Moreover, $\tl \ell_\a$, $\a>0$, is not conjugated in $\tsl$ with $\tl\ell_{-\a}$, thus the normalizer of $\tl\ell_\a$ is the same as its centralizer. We conclude $\Aut^+(\tp/\lg\tl\ell_\a\rg)=\{\tl\ell_\b\ ,\ \b\in\R\}/\lg\tl\ell_\a\rg\simeq\R/\a\Z\simeq S^1$, for all $\a\in\R\sm\pi\Z$.

For $\tl A=\tl P_n^\pm$, $n\in\N$, then again the normalizer of $\lg\tl A\rg$ and its centralizer coincides, as $\tl A$ is not conjugated in $\tsl$ with $\tl A^{-1}$. But $\hat A:=\pj(\tl A)\in\PSL(2,\R)$ has a unique fixed point $[1:0]\in\rp^1$, hence every element $\hat B$ of $\PSL(2,\R)$ that commutes with $\hat A$ must fix this point as well (and no other, unless $\hat B=\id$). This means that $\hat B=\hat P(t)$, $t\in\R$, where the parabolic matrix $P(t)$ is defined in \eqref{hypara}. The canonical lift $\tl P_0(t)$ commutes with $\tl A_0$, the canonical lift of $\hat A$, and because $\tl A=\tl A_n=\tl A_0\ell_{n\pi}$ and $\tl B_m=\tl B_0\ell_{m\pi}$ also commute, we conclude 
$$\Aut^+(\tp/\lg\tl P_n^\pm\rg)=\{\tl P_m(t)\ ,\ m\in\Z,\ t\in\R\}/\lg\tl P_n(\pm1)\rg,$$
which is a semidirect product of $\Aut_0(\tp/\lg\tl P_n^\pm\rg)\simeq\R$ with $\Z/n\Z$. $\Aut_0\tp/\lg\tl P_n^\pm\rg)$ consists of the canonical lifts $\tl P_0(t)$, $t\in\R$, hence all its elements fix all fixed points of $\tl A_0$. There are exactly $n$ such points in $C=\tp/\lg\tl P_n^\pm\rg$.

For $\tl A=\tl\H_n(\l)$, $n\in\N^*$, $\tl A$ is not conjugated in $\tsl$ to $\tl A^{-1}$ hence we again need to determine the centralizer of $\tl A$. As in the elliptic case, if $\tl B$ commutes with $\tl A$, then $B:=\bar\pj(\tl B)\in\SL(2,\R)$ commutes with $A:=\bar\pj(\tl A)$ which has two real eigenvectors $(1,0)$ and $(0,1)$ in $\R^2$, and this means that $B$ has the same real eigenvectors, hence $B=H(\mu)$, $\mu\ne 0$. As in the parabolic case, the canonical lift of $B$ to $\tsl$ commutes with $\tl A_0$ and in fact all lifts $\tl B_n=\tl B_0\tl\ell_{m\pi}$, $\m\in\Z$, of $B$ commute with $\tl A$. We conclude
$$\Aut^+(\tp/\lg\tl H_n(\l)\rg)=\{\tl H_m(\mu)\ ,\ m\in\Z,\ \mu>0\}/\lg\tl H_n(\l)\rg,$$
which is again a semidirect product of $\Aut_0(\tp/\lg\tl H_n(\l)\rg)\simeq (0,\infty)$ with $\Z/n\Z$. 

As $\Aut_0(\tp/\lg\tl H_n(\l)\rg)$ consists of $\tl H_0(\mu)$, $\mu>0$, which all fix the set $\tfrac\pi2\Z$ of fixed points of $\tl A_0$ in $\tp$, we have that the set of $2n$ points on $C=\tp/\lg\tl H_n(\l)\rg$ coming from $\tfrac\pi2\Z$ are fixed by all elements of $\Aut_0(\tp/\lg\tl H_n(\l)\rg)$. 

Finally, for $\tl A=\tl P_0^-$ or $\tl H_0(\l)$, $\Aut^+(C)$ is a quotient by $\lg\tl A\rg$ of the automorphism group of the image of the developing map, which is $(0,\pi)$ in the first case and $(0,\pi/2)$ in the second case. Now, $\Aut^+(0,\pi)=\{\tl P_0(t)\ ,\ t\in\R\}\simeq \R$ hence $\Aut^+(C)\simeq \R/\Z\simeq S^1$ and, in the second case, $\Aut^+(0,\pi/2)=\{\tl\hyp_0(\mu)\ ,\ \mu>0\}\simeq (0,\infty)$, thus 
$\Aut^+(C)\simeq S^1$ as well. In both cases the action of the (connected) group of oriented projective automorphisms is free.
\end{proof}
\begin{ede}\label{reson} A point on a closed projective curve $c$ which is fixed by all elements of $\Aut_0(C)$ is called a {\em resonance point} of $C$.\end{ede}
In the classes $\tl \Para_n^\pm$ or $\tl \Hyp_n(\l)$, $n\in \N^*$, there are $n$, resp. $2n$ resonance points, moreover the set of these points is preserved by all projective automorphisms (preserving the orientation or not, see below), hence the action of the full automorphism group is not transitive; the projective curve is not {\em globally} homogeneous.

Recall, however, that all projective curves are {\em locally} homogeneous manifolds, there is no way to distinguish two points of $C$ by looking at a neighborhood of each point; the existence, on some projective curves $C$, of points that are fixed (or simply permuted) by all projective automorphisms seems to be rather a {\em resonance} phenomenon, hence the use of this term in Definition \ref{reson}.
\begin{remark} 
The infinitesimal automorphisms  of a non-homogeneous projective curve are {\em projective vector fields} (in fact, there is only one up to scale) that vanish at the finite set of resonance points. This vanishing is at order 1 for the classes $\tl\Hyp_n(\l)$, $n\in\N^*$, and at order 2 for the classes $\tl\Para_n^\pm$, $n\in\N^*$. On the projective classes $\tl\Para_0^-$, $\tl\Hyp_0(\l)$ and $\tl\Ell^0$ there is again only one projective vector field (up to scale), and it is nowhere vanishing. On a curve of class $\tl\Ell_{n\pi}$, i.e. on the $n$-fold cover of $\rp^1$, there are projective vector fields of all the kinds described above (if such a vector field vanishes, then it either vanishes at order 2 at $n$ points, or it vanishes at order 1 at $2n$ points), but the zeros of such a vector field $V$ vary depending on $V$, cf. \cite{hit}. \end{remark}

\begin{remark}\label{Aut}
As we have seen in Proposition \ref{orrev}, all closed curves admit orientation-reversing automorphisms, so $\Aut(C)$ is a semidirect product of $\Aut^+(C)$ with $\Z/2\Z$.
\end{remark}

We see thus that all homogeneous projective curves (i.e, the classes
$\tl\Ell\cup\tl\Hyp_0\cup\tl\Para^-_0$) have connected orientation-preserving
automorphism group (isomorphic to $S^1$ except for the classes in
$\tl\Ell_{n\pi},\ n\in\N^*$, where we get a $n$-fold Galois covering of
$\PSL_2(\R)$). 

We summarize in the following table the geometric properties of each projective type of a given curve $C$: 
\begin{itemize}[nosep]
\item the winding number $W(\tilde C)$ of its universal cover $\tilde C\subset \tp$;
\item the winding number $W(C)$ of $C$; 
\item the winding number of the fundamental domains of $C$ (see Proposition \ref{windnumbers})
\item the connected component of the identity in the automorphism group of $C$; 
\item the number of resonance points when $C$ is non-homogeneous / the compatible global affine connections (when they exist); 
\item (up to conjugation and sign) the projection in $\SL(2,\R)$ of the positive generator of the holonomy of $C$.
\end{itemize}

\begin{center}
\begin{adjustbox}{max width=\textwidth}
\begin{tabular}{|l|c|c|c|c|c|c|}
\multicolumn{7}{c}{Projective classes of curves and their
  geometric properties}\medskip \\
\hline
 Projective &    &    &  &  & \!{}$\#$ dist.\!  & $\bar\pj(\tl A)\in \SL(2,\R)$\\
type of $C$  &\!{}$W(\tilde C)$\!   &  $W(C)$  & $W(U)$ &  $\mathrm{Aut}_0(C)$   & points /   & $\tl A$ positive generator\\
&&&&& affine &of the holonomy\\
 \hline\hline
$\tl \Para^-_0$     &   $1$ &   $\tfrac12$  & $\tfrac12$ &  $\R/\Z$  & $\nb^0$    &    $P^-=\begin{pmatrix}1& - 1\\0&1 \end{pmatrix}$  \\
\hline
$ \tl\Para^+_n, \ n\in\N^*$     &   $\infty$ &   $n$  & $n-\tfrac12,n^*$ &  $\R$  &  $n$   &$P^+=\begin{pmatrix}1&1\\0&1\end{pmatrix}$  \\
  \hline
  $ \tl\Para^-_n, \ n\in\N^*$     &   $\infty$ &   $n$  & $n^*,n+\tfrac12$ &  $\R$  &  $n$   &$P^-=\begin{pmatrix}1&- 1\\0&1\end{pmatrix}$  \\
\hline
$ \tl\Hyp_0(\lambda)$     &   $\tfrac12$ &   $\tfrac12$  &  $\tfrac12$ & $\R/\Z$  &   $\nb^+$, $\nb^-$  &$\tl H(\l)=\begin{pmatrix}\l&0\\0&\l^{-1}\end{pmatrix}$   \\
\hline
$ \tl\Hyp_n(\lambda), {n\in\N^*}\!\! $     &   $\infty$ &   $n$  &\!{ $n-\tfrac12,n^*,n+\tfrac12$}\! &  $\R$  &   $2n$  &$H(\l)=\begin{pmatrix}\l&0\\0&\l^{-1}\end{pmatrix}$   \\
\hline
$ \tl\Ell_{\alpha}$,
  $\a\not\in \pi\N$     &
                          $\infty$ &
                                     $[\tfrac\a\pi]+\tfrac12$  & $[\tfrac\a\pi]+\tfrac12$  &
 $\R/\Z$  & &\!$E_\a\!{}=\!\begin{pmatrix}{ \cos\a}&-\sin\a\\ \sin\a&\cos\a\end{pmatrix}$\!   \\
\hline
 $\tl\Ell_{n\pi}, \ n\in\N^*$     &   $\infty$ &   $n$  & $n$  &
\!{\small $\tsl/\langle\tl E_{n\pi}\rangle$}\!  &     &    $\id=\begin{pmatrix}1&0\\0&1\end{pmatrix}$  \\
\hline
\end{tabular}
\end{adjustbox}
\end{center}
Here the asterisk denotes the winding number of the (countably many) particular choices of connected fundamental domains $U$ of $\tl C\subset \tp$ starting (and ending) at some preimage through $\pj$ of one of the resonance points of $C$, as it follows from Proposition \ref{windnumbers}.

\section{The Yamabe problem for curves in conformal geometry} For a compact conformal manifold $(M,c)$ of dimension at least 2, the Yamabe problem asks for a metric in the conformal class that has constant scalar curvature. Notably, the solution to this problem involves the {\em Yamabe operator}, which is precisely the canonical Laplace structure $\LL$ induced on the manifold itself by its conformal structure \cite[Proposition 3.7]{mlc}. This Yamabe operator, or {\em conformal Laplacian}, is a second-order linear differential operator on some (rank 1) weight bundle whose symbol is the contraction with the conformal structure, there is no first order term, and the zero order term, written in a gauge defined by a metric $g$ compatible with $c$, is, up to a constant factor, exactly the scalar curvature $\Scal_g$.

One can reformulate the Yamabe problem (which does admit solutions) as a search for a metric gauge in which the zero order term of the operator $\LL$ is constant.

In view of the above, D. Calderbank and F. Burstall formulate in \cite{cal-bur} the {\em Yamabe problem for curves} in a conformal/Möbius ambient space:

{\em Let $C$ be a curve in a conformal/Möbius manifold $(M,c)$. Is it possible to find a metric in the conformal class of the ambient space such that the zero order term of the induced Laplace structure on $C$, seen in an arc-length parametrization, is constant?}

As we announced in the introduction, the answer is {\em No!}
\begin{prop}\label{yam} There exist curves $C$ in ambient conformal/Möbius spaces $(M,c)$ for which the Yamabe problem has no solution.\end{prop}
\begin{proof} 
In our notation from \eqref{eqh}, $F^\c$ should be constant for $\c$ an arc-length parametrization of $C$. If this is possible, then the projective structure of $C$ has to be homogeneous (the parametrization can be composed with translations, still remain arc-length, and the resulting zero order term is unchanged, see Proposition \ref{la-hi}). 

But we know that there exist (a lot of) projective structures on a closed curve $C$ that are not homogeneous. On the other hand, \cite[Theorem 4.22] {mlc} states, for the particular case of a curve:
\begin{prop} Let $(C,g)$ be a Riemannian curve and $\mathcal{L}$ a Laplace structure on $C$. Let $\nu$ be a vector bundle over $C$. Then there exists a metric $g^\nu$ on the total space $M$ of $\nu$ (and additionally a Möbius structure on $(M,[g])$ if $\dim M=2$) such that $\mathcal{L}$ is the induced Laplace structure on $C$ induced by the ambient space $M$.
\end{prop}
That means, every Laplace structure (hence, every projective structure, including non-homogeneous ones) on $C$ can be realized as one induced by the embedding of $C$ in a conformal/Möbius manifold $M$.

Such a non-homogeneous Laplace structure can not induce, in any parametrization, a Laplace structure with constant zero order term.
\end{proof}

%{\bf Acknowledgment.} The authors are grateful to Paul Gauduchon, whose work (especially the unpublished work cited here) motivated us to develop the theory contained in this paper, which we dedicate to his 80th anniversary.

\end{document}